\documentclass[11pt]{article}

\usepackage[top=50mm, bottom=50mm, left=50mm, right=50mm]{geometry}
\usepackage{lineno}
\usepackage{amssymb}
\usepackage{amsmath}
\usepackage{amsthm}
\usepackage{amsfonts}
\usepackage{epsfig}
\usepackage{graphicx}
\usepackage{graphics}
\usepackage{float}
\usepackage{multirow}
\usepackage{lineno}
\usepackage{fullpage}
\usepackage[normalem]{ulem} 
\usepackage{makeidx}
\usepackage{xspace}
\usepackage{wrapfig}
\usepackage{color}

\makeindex

\theoremstyle{definition}
\newtheorem{theorem}{Theorem}
\newtheorem{defi}{Definition}

\newtheorem{lem}{Lemma}

\numberwithin{defi}{section} 
\numberwithin{theorem}{section}  
\numberwithin{lem}{section}
\numberwithin{equation}{section} 
\newtheorem*{asm}{Assumption}

\newcommand{\Pro}{{\mathbf{P}}}

\newcommand{\Pas}{{\Pro\text{-a.s.}}}
\newcommand{\Fi}{{(\mathcal{F}_t)_{t\ge0}}}

\allowdisplaybreaks[4]

\begin{document}

\title{Global solutions of stochastic nonlinear Schr\"odinger system with quadratic interaction}
 
\author{
{Masaru Hamano\footnote{Faculty of Science and Engineering, Waseda University, Tokyo 169-8555, Japan, email: m.hamano3@kurenai.waseda.jp}} \and {Shunya Hashimoto\footnote{Department of Mathematics, Faculty of Science, Saitama University, Saitama 338-8570, Japan, email: s.hashimoto.230@ms.saitama-u.ac.jp}}  \and {Shuji Machihara\footnote{Department of Mathematics, Faculty of Science, Saitama University, Saitama 338-8570, Japan, email: machihara@rimath.saitama-u.ac.jp}}
}

\date{}

\maketitle

\begin{abstract}
The time-global existence of solutions to a system of stochastic Schr\"odinger equations with multiplicative noise and the quadratic nonlinear terms are discussed in this paper.
The same system in the deterministic treatment was studied in \cite{HOT13} where the mass and energy are conserved. In our stochastic situation, those are not conserved and which causes several difficulties in the arguments of composing a-priori estimate.
\end{abstract}

\section{Preliminaries}
\subsection{Introduction}

We consider the Cauchy problem for the stochastic nonlinear Schr\"odinger system (SNLSS) with multiplicative noise in the general spatial dimension $d\in\mathbb{N}$:
\begin{eqnarray}
\label{SNLSS}
\begin{cases}
idu(t,\xi)=\frac{1}{2\ell}\Delta u(t,\xi)dt+\lambda v(t,\xi)\overline{u(t,\xi)}dt \\
\hspace{5em} -i\mu(\xi)u(t,\xi)dt+iu(t,\xi)dW(t,\xi), \quad t\in(0,T), \ \xi\in \mathbb{R}^d, \\
idv(t,\xi)=\frac{1}{2L}\Delta v(t,\xi)dt+\kappa u^2(t,\xi)dt \\
\hspace{5em} -i\mu(\xi)v(t,\xi)dt+iv(t,\xi)dW(t,\xi), \quad t\in(0,T), \ \xi\in \mathbb{R}^d, \\
u(0,\xi)=u_0(\xi), \quad v(0,\xi)=v_0(\xi), \quad \xi\in\mathbb{R}^d,
\end{cases}
\end{eqnarray}
where $\ell,L>0, \ \lambda,\kappa\in\mathbb{C}.$ For $\{ \mu_j\}_{j=1}^N\subset \mathbb{C}$, the Wiener process $W(t,\xi)$ and the function $\mu$ are given by
\begin{eqnarray}
W(t,\xi)=\sum_{j=1}^N\mu_je_j(\xi)\beta_j(t), \quad t\ge0, \ \xi\in\mathbb{R}^d, \\
\mu(\xi)=\frac{1}{2}\sum_{j=1}^N|\mu_j|^2e_j^2(\xi), \quad \xi\in\mathbb{R}^d.
\end{eqnarray}
Here, $N\in \mathbb{N}\cup \{ +\infty \}$ and the 
$e_j(\xi)$ are real-valued functions.  
The $\beta_j(t)$ are real-valued independent Brownian motions with respect to a 
probability space $(\Omega,\mathcal{F},\Pro)$ with natural filtration 
$(\mathcal{F}_t)_{t\ge0}, \ 1\le j\le N$. In this paper, we assume $N<\infty$ which is the same setting with \cite{BRZ14,BRZ16}. But our techniques easily go over to the case where $N=+\infty$ (i.e. infinite dimensional noise).
We refer to \cite[Remark 2.3.13]{Z14} for details. 

In physics, the nonlinear Schr\"odinger system is an important model, appearing in many physics fields, notably Bose-Einstein condensation. The nonlinear Schr\"odinger system can describe the propagation of wave functions with interactions between two components, as well as the spin and motion of particles (see \cite{BC05,BCGMW97}). The nonlinear Schr\"odinger system is also an important model for nonlinear optics (see \cite{AA99}). In many cases, spatial and temporal fluctuations of the parameters of the medium must be taken into account. This is often caused by random potentials or describes the propagation of dispersive waves in a non-homogeneous or random medium. In such cases, the stochastic Schr\"odinger system is introduced. For more physical interpretations, see \cite{BCRG94,BCRG95,BG09,BH95,BPZ98,FM13,H96} and references therein.

The deterministic case (i.e. $\mu_j=0, \ 1\le j\le N$) was studied by N. Hayashi, T. Ozawa and  K. Tanaka \cite{HOT13}. 
They consider the Cauchy problem for the nonlinear Schr\"odinger system:
\begin{align*}
\begin{cases}
i\partial_tu(t,\xi)+\frac{1}{2\ell}\Delta u(t,\xi)=\lambda v(t,\xi)\overline{u(t,\xi)} \quad (t,x)\in\mathbb{R}\times\mathbb{R}^d, \\
i\partial_tv(t,\xi)+\frac{1}{2L}\Delta v(t,\xi)=\kappa v^2(t,\xi) \quad (t,x)\in\mathbb{R}\times\mathbb{R}^d, \\
u(0,\xi)=u_0(\xi), \quad v(0,\xi)=v_0(\xi), \quad \xi\in\mathbb{R}^d.
\end{cases}
\end{align*}
Under assumption \eqref{49}, the mass and energy in this equation are given by
\begin{align*}
\textbf{mass:} \quad Q(u,v):=& \ ||u||^2_{L^2}+c||v||^2_{L^2}, \\
\textbf{energy:} \quad E(u,v):=& \ K(u,v)+\lambda P(u,v),
\end{align*}
where
\begin{align*}
K(u,v):=& \ \frac{1}{2\ell}||\nabla u||^2_{L^2}+\frac{c}{4L}||\nabla v||^2_{L^2}, \\
P(u,v):=& \ \text{Re}\int_{\mathbb{R}^d} u^2\overline{v} d\xi.
\end{align*}
In \cite{HOT13}, they show the existence of $L^2$-global solutions for $1\le d\le 3$ and the existence of $H^1$-global solutions for $1\le d\le 4$ (with an additional assumption when $d=4$). 

In the single stochastic nonlinear Schr\"odinger equation, V. Barbu, M. R$\ddot{\text{o}}$ckner and D. Zhang \cite{BRZ14,BRZ16} show the well-posedness of $L^2$ and $H^1$ in the same range of exponents as in the deterministic case, using the rescaling transformation introduced in V. Barbu, G. Da Prato and M. R$\ddot{\text{o}}$ckner \cite{BDR09}. 
This transformation reduces the stochastic nonlinear Schr\"odinger equation to an equivalent random Schr\"odinger equation by assuming a decay condition ((H1$)_s$ below) on the noise. 
For more details, see \cite{BRZ17} and references therein.

Also, the stochastic nonlinear Schr\"odinger system with general power-type nonlinear terms with exponent $2\sigma+1$ was studied by Y. Chen, J. Duan and Q. Zhang \cite{CDZ20}.
Namely,
\begin{align*}
\begin{cases}
idu+(\Delta u+(\lambda_{11}|u|^{2\sigma}+\lambda_{12}|v|^{\sigma+1}|u|^{\sigma-1})u)dt=u\circ \phi_1dW(t), \\
idv+(\Delta v+(\lambda_{21}|v|^{\sigma-1}|u|^{\sigma+1}+\lambda_{22}|u|^{2\sigma})v)dt=v\circ \phi_2dW(t), \\
u(0,\xi)=u_0(\xi), \quad v(0,\xi)=v_0(\xi),
\end{cases}
\end{align*}
where the coefficients $\lambda_{ij}\in \mathbb{R}$ for $i,j=1,2, \ (W(t))_{t\ge0}$ is a cylindrical Wiener process in $L^2(\mathbb{R}^N)$ with filtration $(\mathcal{F}_t)_{t\ge0}$, the notation $\circ$ stands for Stratonovich product in the right-hand side, and $\phi_1, \phi_2$ are Hilbert-Schmidt operators from $L^2(\mathbb{R}^N)$ into $H^1(\mathbb{R}^N)$.
In \cite{CDZ20}, using the method by A. de Bouard and A. Debussche \cite{BD99,BD03}, they show the $H^1$- local well-posedness for $\sigma\in [0,\frac{2}{d})\cup (\frac{1}{2},\frac{2}{(d-2)^+})$, the $H^1$- global well-posedness for $\sigma\in [0,\frac{2}{d}]$ (with an additional assumption when $\sigma=\frac{2}{d}$) and the existence of blow-up solutions for $\sigma\in(\frac{2}{d},\frac{2}{(d-2)^+})$. \\

%In this paper, we apply the rescaling transformation. 
%We show the existence of an $L^2$-global solution in (\ref{SNLSS}) for $1\le d\le 3$ and the existence of an $H^1$-global solution in (\ref{SNLSS}) for $1\le d\le 3$ for a rescaled random Schr\"odinger system (\ref{RSNLSS}) assuming a decay condition ((H1$)_s$ below) on the noise.
%There is a significant difference between the proofs of the single probability Schr\"odinger equation and the system probability Schr\"odinger equation.
%It is the existence of an interaction term between $u$ and $v$.
%This makes the estimation of the norm more difficult than in the single stochastic Schr\"odinger equation.
%For this reason, we impose the additional assumption $\text{Re}\mu_j=0$ to show the uniform boundedness of the $H^1$-global solution. \\

\textbf{Theorems in this paper} 

In this paper, we apply the resealing transformation. 
\begin{itemize}
\item The existence of an $L^2$-local solution in (\ref{SNLSS}) for $1\le d\le 4$ (Theorem \ref{main} below).
\item The existence of an $H^1$-local solution in (\ref{SNLSS}) for $1\le d\le 3$ (Theorem \ref{mainH} below).
\item The existence of an $L^2$-global solution in (\ref{SNLSS}) for $1\le d\le 3$ (Theorem \ref{th41} below).
\item The existence of an $H^1$-global solution in (\ref{SNLSS}) for $1\le d\le 3$ (Theorem \ref{th53} below).
\end{itemize}
The proof is based on \cite{BRZ14,BRZ16}.
One significant difference in the proofs of the stochastic Schr\"odinger equation and the stochastic Schr\"odinger system is the existence of interaction terms between $u$ and $v$.
This difference is a major impediment to showing the uniform boundedness of the $H^1$-global solution.
For this reason, we impose the additional assumption \eqref{H1add} to show the uniform boundedness of the $H^1$-global solution.
This assumption is well known in the study of the stochastic Schr\"odinger equation as a condition for the conservation of mass.
However, the existence of $H^1$-global solutions that do not require up to uniform boundedness can be proved without imposing the assumption \eqref{H1add} (Theorem \ref{th541} below).
This proof is used for the argument of the persistence of regularity. 

%To apply the rescaling transformation, we assume the following decay conditions on $(e_j)_{1\le j\le N}$.
%\begin{asm}
%\textbf{(H1$)_s$} \ Let $s\ge0$. Assume that $e_j \in C^{\infty}_b(\mathbb{R}^d)$ such that
%\[ \lim_{|\xi|\to \infty}\zeta(\xi)|\partial^{\gamma}e_j(\xi)|=0, \]
%\qquad where $\gamma$ is a multi-index such that $|\gamma|\le s+2, \ 1\le j\le N$ and 
%\begin{eqnarray*}
%\zeta (\xi)=
%\begin{cases}
%1+|\xi|^2, & \text{if} \ d\not= 2; \\
%(1+|\xi|^2)(\log (3+|\xi|^2))^2, & \text{if} \ d=2.
%\end{cases}
%\end{eqnarray*}
%\end{asm}
We introduce the definition of $H^s$-solution. We use it with $s=0$ or $1$. $H^s$-norm is the standard one, that is,
\[ \|u\|^2_{H^s}=\sum_{|\gamma|\le s}\|D^{\gamma}u\|^2_{L^2}, \]
where $\gamma$ is a multi-index and $D$ is a partial differential operator in $\xi$.
\begin{defi}[$H^s$-solution]
\label{uvdef}
Let $u_0,v_0$ belong to $H^s$. Fix $0<T<\infty$. We say that a triple $(u,v,\tau)$ is a $H^s$-solution of \textnormal{(\ref{SNLSS})}, where $\tau(\le T)$ is an $(\mathcal{F}_t)$-stopping time, and $u=(u(t))_{t\in [0,\tau]},v=(v(t))_{t\in [0,\tau]}$ is an $H^s$-valued continuous $(\mathcal{F}_t)$-adapted process, such that $v\overline{u},u^2\in L^1(0,\tau;H^{s-2})$, $\Pas$, and it satisfies $\Pas$
\begin{align}
\label{14}
u(t)&=u_0-\int_0^t(i\frac{1}{2\ell}\Delta u(r)+\mu u(r)+i\lambda v(r)\overline{u(r)})dr+\int_0^tu(r)dW(r), \quad t\in[0,\tau], \\
v(t)&=v_0-\int_0^t(i\frac{1}{2L}\Delta v(r)+\mu v(r)+i\kappa u^2(r))dr+\int_0^tv(r)dW(r), \quad t\in[0,\tau],
\end{align}
as an equation in $H^{s-2}$.
\end{defi}
We say that uniqueness for (\ref{SNLSS}) holds
in the function space $S$ with respect to variable $\xi$, if for any two $H^s$-solutions $(u_i,v_i,\tau_i), \ u_i,v_i\in S, \ i=1,2$, it holds $\Pas$ that $u_1=u_2,v_1=v_2 \ \text{on} \ [0,\tau_1\wedge \tau_2]$. 

To make the stochastic system with white noise (\ref{SNLSS}) into a random system without white noise, we consider the rescaling transformation
\begin{align}
\label{21}
u(t,\xi)=e^{W(t,\xi)}y(t,\xi), \\
\label{22}
v(t,\xi)=e^{W(t,\xi)}z(t,\xi).
\end{align}
They are a generalization of the Doss-Sussman type transformation in stochastic differential equations, called the method of rescaling (developed by \cite{BDR09}), to infinite dimensions.
Applying \eqref{21} and \eqref{22} to \eqref{SNLSS}, we get

%In this subsection, we introduce a generalization of the Doss-Sussman type transformation in stochastic differential equations, called the method of rescaling (developed by \cite{BDR09}), to infinite dimensions. This makes the stochastic system with white noise (\ref{SNLSS}) into a random system without white noise, which allows the well-posedness of the solution of (\ref{SNLSS}) to be obtained by the same arguments as for the deterministic system.

%We apply the rescaling transformation
%\begin{align}
%\label{21}
%u(t,\xi)=e^{W(t,\xi)}y(t,\xi), \\
%\label{22}
%v(t,\xi)=e^{W(t,\xi)}z(t,\xi).
%\end{align}

%We apply (\ref{21}) and (\ref{22}) to (\ref{SNLSS}) to have
\begin{eqnarray}
\label{RSNLSS}
\begin{cases}
\displaystyle \frac{\partial y(t,\xi)}{\partial t}=A_{\ell}(t)y(t,\xi)-i\lambda\overline{e^{W(t,\xi)}}z(t,\xi)\overline{y(t,\xi)}, \\
\displaystyle \frac{\partial z(t,\xi)}{\partial t}=A_L(t)z(t,\xi)-i\kappa e^{W(t,\xi)}y^2(t,\xi), \\
y(0,\xi)=u_0(\xi), \quad z(0,\xi)=v_0(\xi),
\end{cases}
\end{eqnarray}
where
\begin{align}
\label{EOl}
A_{\ell}(t)y(t,\xi)=&-i\frac{1}{2\ell}e^{-W}\Delta (e^Wy)-(\mu+\tilde{\mu})y \nonumber \\
=&-i(\frac{1}{2\ell}\Delta +b_{\ell}(t,\xi)\cdot \nabla +c_{\ell}(t,\xi))y(t,\xi), \\
\label{EOL}
A_{L}(t)z(t,\xi)=&-i\frac{1}{2L}e^{-W}\Delta (e^Wz)-(\mu+\tilde{\mu})z \nonumber \\
=&-i(\frac{1}{2L}\Delta +b_L(t,\xi)\cdot \nabla +c_L(t,\xi))z(t,\xi), \\
\tilde{\mu}(\xi)=&\frac{1}{2}\sum_{j=1}^N\mu^2_je_j^2(\xi), \\
\label{bsys}
b_{\ell}(t,\xi)=&\frac{1}{\ell}\nabla W(t,\xi), \\
b_{L}(t,\xi)=&\frac{1}{L}\nabla W(t,\xi), \\
\label{csys}
c_{\ell}(t,\xi)=&\frac{1}{2\ell}\sum_{j=1}^d(\partial_jW(t,\xi))^2+\frac{1}{2\ell}\Delta W(t,\xi)-i(\mu(\xi)+\tilde{\mu}(\xi)), \\
c_{L}(t,\xi)=&\frac{1}{2L}\sum_{j=1}^d(\partial_jW(t,\xi))^2+\frac{1}{2L}\Delta W(t,\xi)-i(\mu(\xi)+\tilde{\mu}(\xi)).
\end{align}
We note that the argument needs the following decay conditions on $(e_j)_{1\le j\le N}$.
\begin{asm}
\textbf{(H1$)_s$} \ Let $s\ge0$. Assume that $e_j \in C^{\infty}_b(\mathbb{R}^d)$ such that
\[ \lim_{|\xi|\to \infty}\zeta(\xi)|\partial^{\gamma}e_j(\xi)|=0, \]
\qquad where $\gamma$ is a multi-index such that $|\gamma|\le s+2, \ 1\le j\le N$ and 
\begin{eqnarray*}
\zeta (\xi)=
\begin{cases}
1+|\xi|^2, & \text{if} \ d\not= 2; \\
(1+|\xi|^2)(\log (3+|\xi|^2))^2, & \text{if} \ d=2.
\end{cases}
\end{eqnarray*}
\end{asm}
We give the definition of $H^s$-solution for the rescaled system.
\begin{defi}[$H^s$-solution]
\label{yzdef}
Let $u_0,v_0$ belong to $H^s$. Fix $0<T<\infty$. We say that a triple $(y,z,\tau)$ is a $H^s$-solution of \textnormal{(\ref{RSNLSS})}, where $\tau(\le T)$ is an $(\mathcal{F}_t)$-stopping time, and $y=(y(t))_{t\in [0,\tau]},z=(z(t))_{t\in [0,\tau]}$ is an $H^s$-valued continuous $(\mathcal{F}_t)$-adapted process, such that $z\overline{y},y^2\in L^1(0,\tau;H^{s-2})$, $\Pas$, and it satisfies $\Pas$
\begin{align}
\label{116}
y(t)&=u_0+\int_0^tA_{\ell}(r)y(r)dr-\int_0^ti\lambda\overline{e^{W(r)}}z(r)\overline{y(r)}dr, \quad t\in[0,\tau], \\
\label{117}
z(t)&=v_0+\int_0^tA_{L}(r)z(r)dr-\int_0^ti\kappa e^{W(r)}y^2(r)dr, \quad t\in[0,\tau],
\end{align}
as an equation in $H^{s-2}$.
\end{defi}
We say that uniqueness for (\ref{RSNLSS}) holds
in the function space $S$ with respect to variable $\xi$, if for any two $H^s$-solutions $(y_i,z_i,\tau_i), \ y_i,z_i\in S, \ i=1,2$, it holds $\Pas$ that $y_1=y_2,z_1=z_2 \ \text{on} \ [0,\tau_1\wedge \tau_2]$.

The next theorem holds by the similar argument with the single case \cite{BRZ14, BRZ16}.
\begin{theorem}
\label{uvyz}
For $s=0,1,$ the following holds.
\begin{enumerate}
\item Let $(y,z,\tau)$ be a $H^s$-solution of (\ref{RSNLSS}) in the sense of Definition \ref{yzdef}. Set $u:=e^Wy,v:=e^Wz$. Then $(u,v,\tau)$ is a $H^s$-solution of (\ref{SNLSS}) in the sense of Definition \ref{uvdef}.
\item Let $(u,v,\tau)$ be a $H^s$-solution of (\ref{SNLSS}) in the sense of Definition \ref{uvdef}. 
Set $y:=e^{-W}u,z:=e^{-W}v$. Then $(y,z,\tau)$ is a $H^s$-solution of (\ref{RSNLSS}) in the sense of Definition \ref{yzdef}.
\end{enumerate}
\end{theorem}
To prove well-posedness, 
we introduce two lemmas for the evolution operators and their Strichartz estimates which were shown in \cite{BRZ14} and \cite{BRZ16} (see also \cite{D96}). In this paper, we consider those in the probability space $(\Omega,\mathcal{F},\Pro)$.
\begin{lem}
\label{evosys}
Assume (H1$)_s$. For $\Pas \ \omega\in\Omega,$ operators $A_{\ell}(t),A_{L}(t)$ defined in (\ref{EOl}) and (\ref{EOL}) generate evolution operators $U_{\ell}(t,r)=U_{\ell}(t,r,\omega),U_{L}(t,r)=U_{L}(t,r,\omega), \ 0\le r\le t\le T,$ respectively in the spaces $H^s(\mathbb{R}^d)$. Moreover, for each $x\in H^s(\mathbb{R}^d),$ processes $[r,T]\ni t \mapsto U_{\ell}(t,r)x\in H^s(\mathbb{R}^d), \ [r,T]\ni t \mapsto U_L(t,r)x\in H^s(\mathbb{R}^d),$
are continuous and $(\mathcal{F}_t)$-adapted, hence progressively measurable with respect to the filtration $(\mathcal{F}_t)_{t\ge r}$.
\end{lem}
\begin{lem}
\label{yzSE}
Assume (H1$)_s$. Then for any $T>0, \ u_0,v_0\in H^s$ and $f,g\in L^{q'_2}(0,T;W^{s,p'_2}),$ solutions of the integral equations
\begin{align}
\label{yzms}
y(t)&=U_{\ell}(t,0)u_0+\int_0^tU_{\ell}(t,r)f(r)dr, \quad 0\le t\le T, \\
z(t)&=U_L(t,0)v_0+\int_0^tU_L(t,r)g(r)dr, \quad 0\le t\le T,
\end{align}
satisfy estimates
\begin{align}
\label{yzSEH^s}
||y||_{L^{q_1}(0,T;W^{s,p_1})}&\le C_T(||u_0||_{H^s}+||f||_{L^{q'_2}(0,T;W^{s,p'_2})}), \\
||z||_{L^{q_1}(0,T;W^{s,p_1})}&\le C_T(||v_0||_{H^s}+||g||_{L^{q'_2}(0,T;W^{s,p'_2})}), 
\end{align}
where $(p_1,q_1)$ and $(p_2,q_2)$ are Strichartz pairs, namely
\begin{eqnarray*}
(p_i,q_i)\in [2,\infty]\times [2,\infty]:\frac{2}{q_i}=\frac{d}{2}-\frac{d}{p_i}, \quad \text{if} \ d\not= 2, \\
(p_i,q_i)\in [2,\infty)\times (2,\infty]:\frac{2}{q_i}=\frac{d}{2}-\frac{d}{p_i}, \quad \text{if} \ d=2.
\end{eqnarray*}
Furthermore, the process $C_t, \ t\ge 0,$ can be taken to be $(\mathcal{F}_t)$-progressively measurable, increasing and continuous.
\end{lem}
At the end of this subsection, we describe the organization of this paper.
In Section 2, we prove the existence of $L^2$-time-local solutions for $1\le d\le 4$. In Section 3, we prove the existence of $H^1$-time-local solutions for $1\le d\le 6$.
Then, in Section 4, we prove the existence of $L^2$-time global solutions for $1\le d\le 3$, and in Section 5, we prove the existence of $H^1$-time global solutions for $1\le d\le 3$ under the assumption (\ref{49}) below.
Note that the uniform boundedness of the $H^1$-time global solution requires the further assumption $\text{Re} \mu_j=0$.

\section{Local existence of $L^2$-solutions}

For any $u_0,v_0\in L^2$, we solve (\ref{SNLSS}) in the spaces
\begin{align}
\label{solspX}
\mathcal{X}_t:=
\begin{cases}
(C\cap L^{\infty})([0,t];L^2)\cap L^4(0,t;L^{\infty}), \quad d=1, \\
(C\cap L^{\infty})([0,t];L^2)\cap L^{q_0}(0,t;L^{r_0}), \quad d=2, \\
(C\cap L^{\infty})([0,t];L^2)\cap L^2(0,t;L^{\frac{2d}{d-2}}), \quad d=3,4,
\end{cases}
\end{align}
where $0<\frac{2}{q_0}=1-\frac{2}{r_0}<1$ with $r_0$ sufficiently large.
\begin{theorem}
\label{main}
Assume (H1$)_0$ and $1\le d\le 4$. Then, for each $u_0,v_0\in L^2$ and $0<T<\infty$, there exists a sequence of $L^2$-local solutions $(u_n,v_n,\tau_n)$ of \textnormal{(\ref{SNLSS})}, $n\in \mathbb{N}$ where $\tau_n$ is a sequence of increasing stopping times. For every $n\ge 1$, it holds $\Pas$ that
\begin{align}
\label{uvn}
u_n|_{[0,\tau_n]},v_n|_{[0,\tau_n]}\in \mathcal{X}_{\tau_n},
\end{align}
and uniqueness holds in the function space 
$\mathcal{X}_{\tau_n}$. Moreover for $1\le d\le 3$, defining $\displaystyle \tau^*(u_0,v_0)=\lim_{n\to \infty}\tau_n, \ u=\lim_{n\to \infty}u_n\mathbf{1}_{[0,\tau^*(u_0,v_0))}$ and $\displaystyle v=\lim_{n\to \infty}v_n\mathbf{1}_{[0,\tau^*(u_0,v_0))},$ we have the blowup alternative, that is, for $\Pas \ \omega$, if $\tau_n(\omega)<\tau^*(u_0,v_0)(\omega),$ \ for every $n\in \mathbb{N}$, then
\begin{align*}
\lim_{t\to \tau^*(u_0,v_0)(\omega)}(||u(t)(\omega)||_{L^2}&+||v(t)(\omega)||_{L^2})=\infty.
\end{align*}
\end{theorem}
By the equivalence of two expressions of solutions via the rescaling transformations (\ref{21}) and (\ref{22}),
Theorem \ref{main} is rewritten as Theorem \ref{yzmain}.
\begin{theorem}
\label{yzmain}
Assume (H1$)_0$ and $1\le d\le 4$. Then, for each $u_0,v_0\in L^2$ and $0<T<\infty$, there exists a sequence of $L^2$-local solutions $(y_n,z_n,\tau_n)$ of \textnormal{(\ref{RSNLSS})}, $n\in \mathbb{N}$ where $\tau_n$ is a sequence of increasing stopping times. For every $n\ge 1$, it holds $\Pas$ that
\begin{align}
\label{yzn}
y_n|_{[0,\tau_n]},z_n|_{[0,\tau_n]}\in \mathcal{X}_{\tau_n},
\end{align}
and uniqueness holds in the function space 
$\mathcal{X}_{\tau_n}$. Moreover for $1\le d\le 3$, defining $\displaystyle \tau^*(u_0,v_0)=\lim_{n\to \infty}\tau_n, \ y=\lim_{n\to \infty}y_n\mathbf{1}_{[0,\tau^*(u_0,v_0))}$ and $\displaystyle z=\lim_{n\to \infty}z_n\mathbf{1}_{[0,\tau^*(u_0,v_0))},$ we have the blowup alternative, that is, for $\Pas \ \omega$, if $\tau_n(\omega)<\tau^*(u_0,v_0)(\omega),$ \ for every $n\in \mathbb{N}$, then
\begin{align*}
\lim_{t\to \tau^*(u_0,v_0)(\omega)}(||y(t)(\omega)||_{L^2}&+||z(t)(\omega)||_{L^2})=\infty.
\end{align*}
\end{theorem}
\begin{proof}[Proof of Theorem \ref{yzmain}]
We solve the weak equations (\ref{116}) and (\ref{117}) in the mild sense, which is equivalent to Definition \ref{yzdef}, namely 
\begin{align}
\label{mild1}
y(t)&=U_{\ell}(t,0)u_0-i\int_0^tU_{\ell}(t,s)(\lambda\overline{e^{W(s)}}z(s)\overline{y(s)})ds, \quad t\in[0,T], \\
\label{mild2}
z(t)&=U_L(t,0)v_0-i\int_0^tU_L(t,s)(\kappa e^{W(s)}y^2(s))ds, \quad t\in [0,T].
\end{align}
We consider the following map
\begin{align}
\label{tauCM}
\mathcal{T}(y,z)(t)=&\left( \Phi(y,z),\Psi(y,z)\right) \nonumber \\
:=&\left( U_{\ell}(t,0)u_0-i\int_0^tU_{\ell}(t,s)(\lambda\overline{e^{W(s)}}z(s)\overline{y(s)})ds, \right. \nonumber \\
&\quad\left. U_L(t,0)v_0-i\int_0^tU_L(t,s)(\kappa e^{W(s)}y^2(s))ds \right).
\end{align}
First, we consider the case $d=1$. Then, the space $\mathcal{X}_t$ is $(C\cap L^{\infty})([0,t];L^2)\cap L^4(0,t;L^{\infty})$. \\
\textbf{Step 1.} First, we find $(y_1,z_1,\tau_1)$. Fix $\omega \in \Omega$ and consider $\mathcal{T}(y,z)$ on the set
\[ E^{\tau_1}_{M_1}=\{ (y,z)\in (\mathcal{X}_{\tau_1})^2:||y||_{\mathcal{X}_{\tau_1}}+||z||_{\mathcal{X}_{\tau_1}}\le M_1\}, \]
where $||y||_{\mathcal{X}_{\tau_1}}:=||y||_{L^{\infty}(0,t;L^2)}+||y||_{L^4(0,t;L^{\infty})}, \ \tau_1=\tau_1(\omega)\in (0,T]$ and $M_1=M_1(\omega)>0$ are random variables. Using the Strichartz estimates and H$\ddot{\text{o}}$lder's inequality, the following hold. 
\begin{align*}
||\Phi(y,z)||_{\mathcal{X}_{\tau_1}}\lesssim& ||u_0||_{L^2}+||z\overline{y}||_{L^1(0,\tau_1;L^2)} \\
\lesssim& ||u_0||_{L^2}+\tau_1^{\frac{3}{4}}||z||_{L^4(0,\tau_1;L^{\infty})}||y||_{L^{\infty}(0,\tau_1;L^2)}, \\
||\Psi(y,z)||_{\mathcal{X}_{\tau_1}}\lesssim&||v_0||_{L^2}+\tau_1^{\frac{3}{4}}||y||_{L^4(0,\tau_1;L^{\infty})}||y||_{L^{\infty}(0,\tau_1;L^2)}.
\end{align*}
Therefore,
\begin{align*}
||\mathcal{T}(y,z)||_{(\mathcal{X}_{\tau_1})^2}\lesssim& ||u_0||_{L^2}+||v_0||_{L^2}+2\tau_1^{\frac{3}{4}}M_1^{2}.
\end{align*}
We shall choose $M_1$ and $\tau_1$ to obtain $\mathcal{T}(E^{\tau_1}_{M_1})\subset E_{M_1}^{\tau_1}$ such that
\[ C_{\tau_1}(||u_0||_{L^2}+||v_0||_{L^2}+2\tau_1^{\frac{3}{4}}M_1^{2})\le M_1. \]
To this end, we define the real-valued continuous, $(\mathcal{F}_t)$-adapted process
\[ Z_t^{(1)}:=3C_t^{2}(||u_0||_{L^2}+||v_0||_{L^2})t^{\frac{3}{4}}, \]
choose the $(\mathcal{F}_t)$-stopping time as
\[ \tau_1:=\inf \left\{ t\in [0,T];Z_t^{(1)}>\frac{1}{3} \right\}\wedge T, \]
and set $M_1=3C_{\tau_1}(||u_0||_{L^2}+||v_0||_{L^2})$, where $a\wedge b:=\min\{a,b\}$. Then it follows that $Z_{\tau_1}^{(1)}\le \frac{1}{3}$ and $\mathcal{T}(E^{\tau_1}_{M_1})\subset E_{M_1}^{\tau_1}$.

Moreover, for $(y,z),(y',z')\in E^{\tau_1}_{M_1}$, we have
\begin{align*}
&||\Phi(y,z)-\Phi(y',z')||_{\mathcal{X}_{\tau_1}} \\
&\le C_{\tau_1}||z\overline{y}-z'\overline{y'}||_{L^1(0,\tau_1;L^2)} \\
&\le C_{\tau_1}\tau_1^{\frac{3}{4}} \left( ||z-z'||_{L^{\infty}(0,\tau_1;L^2)}||y||_{L^4(0,\tau_1;L^{\infty})}+||z'||_{L^4(0,\tau_1;L^{\infty})}||y-y'||_{L^{\infty}(0,\tau_1;L^2)} \right) \\
&\le C_{\tau_1}\tau_1^{\frac{3}{4}}(||y||_{L^4(0,\tau_1;L^{\infty})}+||z'||_{L^4(0,\tau_1;L^{\infty})})(||y-y'||_{L^{\infty}(0,\tau_1;L^2)}+||z-z'||_{L^{\infty}(0,\tau_1;L^2)}) \\
&\le 2C_{\tau_1}M_1\tau_1^{\frac{3}{4}} \left( ||y-y'||_{L^{\infty}(0,\tau_1;L^2)}+||z-z'||_{L^{\infty}(0,\tau_1;L^2)} \right) \\
&\le \frac{1}{3} \left( ||y-y'||_{L^{\infty}(0,\tau_1;L^2)}+||z-z'||_{L^{\infty}(0,\tau_1;L^2)} \right),
\end{align*}
\begin{align*}
||\Psi(y,z)-\Psi(y',z')||_{\mathcal{X}_{\tau_1}}\le& \frac{1}{3} \left( ||y-y'||_{L^{\infty}(0,\tau_1;L^2)}+||z-z'||_{L^{\infty}(0,\tau_1;L^2)} \right).
\end{align*}
Therefore,
\begin{align*}
||\mathcal{T}(y,z)-\mathcal{T}(y',z')||_{(\mathcal{X}_{\tau_1})^2}\le\frac{2}{3} \left( ||y-y'||_{L^{\infty}(0,\tau_1;L^2)}+||z-z'||_{L^{\infty}(0,\tau_1;L^2)} \right),
\end{align*}
which implies that $\mathcal{T}$ is a contraction mapping on $(\mathcal{X}_{\tau_1})^2$. 
Since $E_{M_1}^{\tau_1}$ is a complete metric subspace in $(\mathcal{X}_{\tau_1})^2$, Banach's fixed point theorem yields a unique $(y,z)\in E_{M_1}^{\tau_1}$ with $\mathcal{T}((y,z))=(y,z) \ \text{on} \ [0,\tau_1]$. 
Therefore, setting $y_1(t):=y(t\wedge \tau_1), \ z_1(t):=z(t\wedge \tau_1), \ t\in [0,T]$, we deduce that $(y_1,z_1,\tau_1)$ is a $L^2$-solution of (\ref{RSNLSS}), such that $y_1(t)=y_1(t\wedge \tau_1), \ z_1(t)=z_1(t\wedge \tau_1), \ t\in [0,T],$ and $y_1|_{[0.\tau_1]},z_1|_{[0,\tau_1]}\in \mathcal{X}_{\tau_1}.$

\textbf{Step 2.} We use an induction argument. Suppose that at the $n$-th step we have a $L^2$-solution $(y_n,z_n,\tau_n)$ of (\ref{RSNLSS}), such that $\tau_n\ge \tau_{n-1}, \ y_n(t)=y_n(t\wedge \tau_n), \ z_n(t)=z_n(t\wedge \tau_n), \ t\in[0,T],$ and $y_n|_{[0,\tau_n]},z_n|_{[0,\tau_n]}\in \mathcal{X}_{\tau_n}.$ We construct $(y_{n+1},z_{n+1},\tau_{n+1})$. Set
\begin{eqnarray*}
E^{\sigma_n}_{M_{n+1}}=\{ (y,z)\in (\mathcal{X}_{\sigma_n})^2;||y||_{\mathcal{X}_{\sigma_n}}+||z||_{\mathcal{X}_{\sigma_n}}\le M_{n+1} \},
\end{eqnarray*}
where $\sigma_n=\sigma_n(\omega)\in (0,T], \ M_{n+1}=M_{n+1}(\omega)>0$ are random variables. 
Then, define the map $\mathcal{T}_n$ by
\begin{align*}
\mathcal{T}_n(y,z)(t)=&\left( \Phi_n(y,z),\Psi_n(y,z)\right) \\
:=&\left( U_{\ell}(\tau_n+t,\tau_n)y_n(\tau_n)-i\int_0^tU_{\ell}(\tau_n+t,\tau_n+s)(\lambda\overline{e^{W(\tau_n+s)}}z(s)\overline{y(s)})ds, \right. \\
&\quad\left. U_L(\tau_n+t,\tau_n)z_n(\tau_n)-i\int_0^tU_L(\tau_n+t,\tau_n+s)(\kappa e^{W(\tau_n+s)}y^2(s))ds \right).
\end{align*}
Analogous calculations as in Step 1. show that for $(y,z)\in E^{\sigma_n}_{M_{n+1}}$
\begin{align*}
||\mathcal{T}_n(y,z)||_{(\mathcal{X}_{\sigma_n})^2}\lesssim& ||y_n(\tau_n)||_{L^2}+||z_n(\tau_n)||_{L^2}+2\sigma_n^{\frac{3}{4}}M_{n+1}^{2}.
\end{align*}
We shall choose $M_{n+1}$ and $\sigma_n$ to obtain $\mathcal{T}_n(E_{M_{n+1}}^{\sigma_n})\subset E_{M_{n+1}}^{\sigma_n}$ such that
\[ C_{\tau_n+\sigma_n}(||y_n(\tau_n)||_{L^2}+||z_n(\tau_n)||_{L^2}+2\sigma_n^{\frac{3}{4}}M_{n+1}^{2}) \le M_{n+1}. \]
To this end, we define the real-valued continuous, $(\mathcal{F}_{\tau_n+t})$-adapted process 
\[ Z_t^{(n+1)}:=3C_{\tau_n+t}^{2}(||y_n(\tau_n)||_{L^2}+||z_n(\tau_n)||_{L^2})t^{\frac{3}{4}}, \]
choose the $(\mathcal{F}_{\tau_n+t})$-stopping time as
\[ \sigma_n=\inf \left\{ t\in[0,T-\tau_n] ;Z_t^{(n+1)}>\frac{1}{3} \right\} \wedge (T-\tau_n), \]
and set $M_{n+1}=3C_{\tau_n+\sigma_n}(||y_n(\tau_n)||_{L^2}+||z_n(\tau_n)||_{L^2})$. Then it follows that $Z_{\sigma_n}^{(n+1)}\le \frac{1}{3}$ and $\mathcal{T}_n(E^{\sigma_n}_{M_{n+1}})\subset E^{\sigma_n}_{M_{n+1}}$. 

Moreover, for $(y,z),(y',z')\in E^{\sigma_n}_{M_{n+1}},$
\begin{align*}
||\mathcal{T}_n(y,z)-\mathcal{T}_n(y',z')||_{(\mathcal{X}_{\sigma_n})^2}\le\frac{2}{3} \left( ||y-y'||_{L^{\infty}(0,\sigma_n;L^2)}+||z-z'||_{L^{\infty}(0,\sigma_n;L^2)} \right).
\end{align*}
which implies that $\mathcal{T}_n$ is a contraction mapping on $(\mathcal{X}_{\sigma_n})^2$.

Set $\tau_{n+1}=\tau_n+\sigma_n$. Then, similarly to the proof of \cite[Lemma 4.2]{BRZ14}, we can show $\tau_{n+1}$ is an $\Fi$-stopping time. 
By Banach's fixed point theorem, there exists a unique $(w_{n+1},x_{n+1})\in E^{\sigma_n}_{M_{n+1}}$ satisfying $\mathcal{T}_n((w_{n+1},x_{n+1}))=(w_{n+1},x_{n+1})$ on $[0,\sigma_n]$. We define
\begin{eqnarray*}
y_{n+1}(t)=
\begin{cases}
y_n(t), & t\in [0,\tau_n]; \\
w_{n+1}((t-\tau_n)\wedge \sigma_n), & t\in (\tau_n,T],
\end{cases}
\\
z_{n+1}(t)=
\begin{cases}
z_n(t), & t\in [0,\tau_n]; \\
x_{n+1}((t-\tau_n)\wedge \sigma_n), & t\in (\tau_n,T].
\end{cases}
\end{eqnarray*}
It follows from the definition of $\mathcal{T}$ and $\mathcal{T}_n$ that $(y_{n+1},z_{n+1})=\mathcal{T}((y_{n+1},z_{n+1}))$ on $[0,\tau_{n+1}]$. 

And, similar to the proof of \cite[Lemma 6.2]{BRZ14}, $y_{n+1}$ and $z_{n+1}$ are an adapted to $\Fi$ in $L^2$. 
Hence, $(y_{n+1},z_{n+1},\tau_{n+1})$ is a $L^2$-solution of (\ref{RSNLSS}), such that $y_{n+1}(t)=y_{n+1}(t\wedge \tau_{n+1}), \ z_{n+1}(t)=z_{n+1}(t\wedge \tau_{n+1}), \ t\in[0,T]$, and $y_{n+1}|_{[0,\tau_{n+1}]},z_{n+1}|_{[0,\tau_{n+1}]}\in \mathcal{X}_{\tau_{n+1}}$. 
Starting from Step 1 and repeating the procedure in Step 2, we finally construct a sequence of $L^2$-solutions $(y_n,z_n,\tau_n), \ n\in \mathbb{N},$ where $\tau_n$ are increasing stopping times and $y_{n+1}=y_n, \ z_{n+1}=z_n \ \text{on} \ [0,\tau_n]$. 

To prove the uniqueness, for any two $L^2$-solutions $(\widetilde{y}_i,\widetilde{z}_i,\sigma_i), \ i=1,2,$ define $\iota=\sup \{ t\in[0,\sigma_1\wedge \sigma_2]:\widetilde{y}_1=\widetilde{y}_2, \ \widetilde{z}_1=\widetilde{z}_2 \ \text{on} \ [0,t] \}$. Suppose that $\Pro (\iota <\sigma_1\wedge \sigma_2)>0$. For $\omega \in \{ \iota<\sigma_1\wedge \sigma_2 \},$ we have $\widetilde{y}_1(\omega)=\widetilde{y}_2(\omega), \ \widetilde{z}_1(\omega)=\widetilde{z}_2(\omega)$ on $[0,\iota(\omega)]$ by the continuity in $L^2$, and for $t\in [0,\sigma_1\wedge \sigma_2(\omega)-\iota(\omega)),$
\begin{eqnarray*}
& & ||\widetilde{y}_1(\omega)-\widetilde{y}_2(\omega)||_{\mathcal{X}(\iota(\omega),\iota(\omega)+t)}+||\widetilde{z}_1(\omega)-\widetilde{z}_2(\omega)||_{\mathcal{X}(\iota(\omega),\iota(\omega)+t)} \\
& = & ||\Phi(\widetilde{y}_1(\omega))-\Phi(\widetilde{y}_2(\omega))||_{\mathcal{X}(\iota(\omega),\iota(\omega)+t)}+||\Psi(\widetilde{z}_1(\omega))-\Psi(\widetilde{z}_2(\omega))||_{\mathcal{X}(\iota(\omega),\iota(\omega)+t)} \\
& \lesssim & C_{\iota(\omega)+t}\widetilde{M}(t)t^{\theta}(||\widetilde{y}_1(\omega)-\widetilde{y}_2(\omega)||_{L^{\infty}(\iota(\omega),\iota(\omega)+t;L^2)}+||\widetilde{z}_1(\omega)-\widetilde{z}_2(\omega)||_{L^{\infty}(\iota(\omega),\iota(\omega)+t;L^2)}),
\end{eqnarray*}
where $\widetilde{M}(t):=||\widetilde{y}_1(\omega)||_{L^4(\iota(\omega),\iota(\omega)+t;L^{\infty})}+||\widetilde{y}_2(\omega)||_{L^4(\iota(\omega),\iota(\omega)+t;L^{\infty})}+||\widetilde{z}_1(\omega)||_{L^4(\iota(\omega),\iota(\omega)+t;L^{\infty})}+$\\
$||\widetilde{z}_2(\omega)||_{L^4(\iota(\omega),\iota(\omega)+t;L^{\infty})} \to 0$ as $t\to 0$. Therefore, with $t$ small enough we deduce that $\widetilde{y}_1(\omega)=\widetilde{y}_2(\omega),\widetilde{z}_1(\omega)=\widetilde{z}_2(\omega)$ on $[\iota(\omega),\iota(\omega)+t]$, hence $\widetilde{y}_1(\omega)=\widetilde{y}_2(\omega),\widetilde{z}_1(\omega)=\widetilde{z}_2(\omega)$ on $[0,\iota(\omega)+t]$, which contradicts the definition of $\iota$.

Next, we consider the case $d=2$. Then, the space $\mathcal{X}_t$ is $(C\cap L^{\infty})([0,t];L^2)\cap L^{q_0}(0,t;L^{r_0}),$ where $0<\frac{2}{q_0}=1-\frac{2}{r_0}<1$ with $r_0$ sufficiently large. We estimate $\Phi(y,z)$ and $\Psi(y,z)$ as
\begin{align*}
||\Phi(y,z)||_{\mathcal{X}_t}\lesssim& ||u_0||_{L^2}+||z\overline{y}||_{L^{q_0'}(0,t;L^{r_0'})}\lesssim ||u_0||_{L^2}+t^{\frac{1}{2}}||z||_{L^{r_0}(0,t;L^{\frac{2r_0}{r_0-2}})}||y||_{L^{\infty}(0,t;L^2)}, \\
&||\Psi(y,z)||_{\mathcal{X}_t}\lesssim||v_0||_{L^2}+t^{\frac{1}{2}}||y||_{L^{r_0}(0,t;L^{\frac{2r_0}{r_0-2}})}||y||_{L^{\infty}(0,t;L^2)}.
\end{align*}
Similarly,
\begin{align*}
&||\Phi(y,z)-\Phi(y',z')||_{\mathcal{X}_t} \\
\le& C_{t}t^{\frac{1}{2}}\left(||y||_{L^{r_0}(0,t;L^{\frac{2r_0}{r_0-2}})}+||z'||_{L^{r_0}(0,t;L^{\frac{2r_0}{r_0-2}})}\right)(||y-y'||_{L^{\infty}(0,t;L^2)}+||z-z'||_{L^{\infty}(0,t;L^2)}), \\
&||\Psi(y,z)-\Psi(y',z')||_{\mathcal{X}_t} \\
\le& C_{t}t^{\frac{1}{2}}\left(||y||_{L^{r_0}(0,t;L^{\frac{2r_0}{r_0-2}})}+||y'||_{L^{r_0}(0,t;L^{\frac{2r_0}{r_0-2}})}\right)(||y-y'||_{L^{\infty}(0,t;L^2)}+||z-z'||_{L^{\infty}(0,t;L^2)}).
\end{align*}
Note that
\begin{align*}
||y||_{L^{r_0}(0,t;L^{\frac{2r_0}{r_0-2}})}\le ||y||^{\frac{2}{r_0-2}}_{L^{q_0}(0,t;L^{r_0})}||y||^{\frac{r_0-4}{r_0-2}}_{L^{\infty}(0,t;L^2)}.
\end{align*}
Next, we consider the case $d=3$. Then, the space $\mathcal{X}_t$ is $(C\cap L^{\infty})([0,t];L^2)\cap L^2(0,t;L^{6})$. We estimate $\Phi(y,z)$ and $\Psi(y,z)$ as
\begin{align*}
||\Phi(y,z)||_{\mathcal{X}_t}\lesssim& ||u_0||_{L^2}+||z\overline{y}||_{L^{\frac{4}{3}}(0,t;L^{\frac{3}{2}})}\lesssim ||u_0||_{L^2}+t^{\frac{1}{4}}||z||_{L^{2}(0,t;L^{6})}||y||_{L^{\infty}(0,t;L^2)}, \\
&||\Psi(y,z)||_{\mathcal{X}_t}\lesssim||v_0||_{L^2}+t^{\frac{1}{4}}||y||_{L^{2}(0,t;L^{6})}||y||_{L^{\infty}(0,t;L^2)}.
\end{align*}
Similarly,
\begin{align*}
&||\Phi(y,z)-\Phi(y',z')||_{\mathcal{X}_t} \\
\le& C_{t}t^{\frac{1}{4}}(||y||_{L^{2}(0,t;L^{6})}+||z'||_{L^{2}(0,t;L^{6})})(||y-y'||_{L^{\infty}(0,t;L^2)}+||z-z'||_{L^{\infty}(0,t;L^2)}), \\
&||\Psi(y,z)-\Psi(y',z')||_{\mathcal{X}_t} \\
\le& C_{t}t^{\frac{1}{4}}(||y||_{L^{2}(0,t;L^{6})}+||y'||_{L^{2}(0,t;L^{6})})(||y-y'||_{L^{\infty}(0,t;L^2)}+||z-z'||_{L^{\infty}(0,t;L^2)}).
\end{align*}
Next, we consider the case $d=4$. Then, the space $\mathcal{X}_t$ is $(C\cap L^{\infty})([0,t];L^2)\cap L^2(0,t;L^{4})$. We estimate $\Phi(y,z)$ and $\Psi(y,z)$ in $L^2(0,t;L^{4})$ as
\begin{align*}
||\Phi(y,z)||_{L^2(0,t;L^{4})}\lesssim& ||U_{\ell}(\cdot,\cdot)u_0||_{L^2(0,t;L^{4})}+||z\overline{y}||_{L^{1}(0,t;L^{2})} \\
\lesssim& ||U_{\ell}(\cdot,\cdot)u_0||_{L^2(0,t;L^{4})}+||z||_{L^{2}(0,t;L^{4})}||y||_{L^{2}(0,t;L^4)}, \\
||\Psi(y,z)||_{L^2(0,t;L^{4})}\lesssim&||U_L(\cdot,\cdot)v_0||_{L^2(0,t;L^{4})}+||y||^2_{L^{2}(0,t;L^{4})}.
\end{align*}
Similarly,
\begin{align*}
&||\Phi(y,z)-\Phi(y',z')||_{L^2(0,t;L^{4})} \\
\le& C_{t}(||y||_{L^{2}(0,t;L^{4})}+||z'||_{L^{2}(0,t;L^{4})})(||y-y'||_{L^{2}(0,t;L^{4})}+||z-z'||_{L^{2}(0,t;L^{4})}), \\
&||\Psi(y,z)-\Psi(y',z')||_{L^2(0,t;L^{4})} \\
\le& C_{t}(||y||_{L^{2}(0,t;L^{4})}+||y'||_{L^{2}(0,t;L^{4})})(||y-y'||_{L^{2}(0,t;L^{4})}+||z-z'||_{L^{2}(0,t;L^{4})}).
\end{align*}
For $u_0,v_0$, we know that $U_{\ell}(\cdot,\cdot)u_0,U_{L}(\cdot,\cdot)v_0\in L^2(0,t;L^4)$, so by taking $T>0$ sufficiently small, the associated norm can be arbitrarily small. Thus, the argument of the contraction mapping works on a closed ball in $L^2(0,t;L^4)$ centered at the origin and of a sufficiently small radius. Then the solution satisfies the integral equations $\mathcal{T}$ and belongs to $\mathcal{X}_t$ by Strichartz estimates.

Finally for $1\le d\le3$, we prove the blowup alternative. Suppose that 
\[\Pro(M^*<\infty;\tau_n<\tau^*(u_0,v_0),\forall n\in\mathbb{N})>0,\] 
where
\[ M^*:=\sup_{t\in[0,\tau^*(u_0,v_0))}(||y(t)||_{L^2}+||z(t)||_{L^2}).\] 
Define
\begin{align*}
Z_t:&=3C^2_{T+t}M^*t^{\frac{3}{4}}, \\
\sigma:&=\inf\{t\in[0,T];Z_t>\frac{1}{6}\}\wedge T.
\end{align*}
For $\omega\in\{M^*<\infty;\tau_n<\tau^*(u_0,v_0),\forall n\in\mathbb{N}\}$, since $\tau_n(\omega)<T, \ \forall n\in \mathbb{N},$ by the definition of $\sigma_n$ in Step 2, we have
\[ \sigma_n(\omega)=\inf \left\{ t\in[0,T-\tau_n(\omega)] ;Z_t^{(n+1)}(\omega)>\frac{1}{3} \right\}. \]
On the other hand, since $||y_n(\tau_n)||_{L^2}+||z_n(\tau_n)||_{L^2}\le M^*$ and $C_{\tau_n+t}\le C_{T+t}$ for any $n\ge1$, $Z_t\ge Z_t^{(n+1)}$ holds. Therefore $\sigma_n(\omega)>\sigma(\omega)>0.$ Hence $\tau_{n+1}(\omega)=\tau_n(\omega)+\sigma_n(\omega)>\tau_n(\omega)+\sigma(\omega)$. This implies $\tau_{n+1}(\omega)>\tau_1(\omega)+n\sigma(\omega)$ for every $n\ge 1$. Therefore, after a finite number of iterations $\tau_n(\omega)$ exceeds $T$. This contradicts $\tau_n(\omega)\le T$. This completes the proof.
\end{proof}

\section{Local existence of $H^1$-solutions}

For any $u_0,v_0\in H^1$, we solve (\ref{SNLSS}) in the spaces
\begin{align}
\label{solspY}
\mathcal{Y}_t:=
\begin{cases}
(C\cap L^{\infty})([0,t];H^1)\cap L^4(0,t;W^{1,\infty}), \quad d=1, \\
(C\cap L^{\infty})([0,t];H^1)\cap L^{q_0}(0,t;W^{1,r_0}), \quad d=2, \\
(C\cap L^{\infty})([0,t];H^1)\cap L^2(0,t;W^{1,\frac{2d}{d-2}}), \quad d=3,4,5,6,
\end{cases}
\end{align}
where $0<\frac{2}{q_0}=1-\frac{2}{r_0}<1$ with $r_0$ sufficiently large.
\begin{theorem}
\label{mainH}
Assume (H1$)_1$ and $1\le d\le 6$. Then, for each $u_0,v_0\in H^1$ and $0<T<\infty$, there exists a sequence of $H^1$-local solutions $(u_n,v_n,\tau_n)$ of \textnormal{(\ref{SNLSS})}, $n\in \mathbb{N}$ where $\tau_n$ is a sequence of increasing stopping times. For every $n\ge 1$, it holds $\Pas$ that
\begin{align*}
u_n|_{[0,\tau_n]},v_n|_{[0,\tau_n]}\in \mathcal{Y}_{\tau_n},
\end{align*}
and uniqueness holds in the function space 
$\mathcal{Y}_{\tau_n}$. Moreover for $1\le d\le 5$, defining $\displaystyle \tau^*(u_0,v_0)=\lim_{n\to \infty}\tau_n, \ u=\lim_{n\to \infty}u_n\mathbf{1}_{[0,\tau^*(u_0,v_0))}$ and $\displaystyle v=\lim_{n\to \infty}v_n\mathbf{1}_{[0,\tau^*(u_0,v_0))},$ we have the blowup alternative, that is, for $\Pas \ \omega$, if $\tau_n(\omega)<\tau^*(u_0,v_0)(\omega),$ for every $n\in \mathbb{N}$, then
\begin{align*}
\lim_{t\to \tau^*(u_0,v_0)(\omega)}(||u(t)(\omega)||_{H^1}&+||v(t)(\omega)||_{H^1})=\infty.
\end{align*}
\end{theorem}
By the equivalence of two expressions of solutions via the rescaling transformations (\ref{21}) and (\ref{22}),
Theorem \ref{mainH} is rewritten as Theorem \ref{ymainH}.
\begin{theorem}
\label{ymainH}
Assume (H1$)_1$, and $1\le d\le 6$. Then, for each $u_0,v_0\in H^1$ and $0<T<\infty$, there exists a sequence of $H^1$-local solutions $(y_n,z_n,\tau_n)$ of \textnormal{(\ref{RSNLSS})}, $n\in \mathbb{N}$ where $\tau_n$ is a sequence of increasing stopping times. For every $n\ge 1$, it holds $\Pas$ that
\begin{align*}
y_n|_{[0,\tau_n]},z_n|_{[0,\tau_n]}\in \mathcal{Y}_{\tau_n},
\end{align*}
and uniqueness holds in the function space 
$\mathcal{Y}_{\tau_n}$. Moreover for $1\le d\le 5$, defining $\displaystyle \tau^*(u_0,v_0)=\lim_{n\to \infty}\tau_n, \ y=\lim_{n\to \infty}y_n\mathbf{1}_{[0,\tau^*(u_0,v_0))}$ and $\displaystyle z=\lim_{n\to \infty}z_n\mathbf{1}_{[0,\tau^*(u_0,v_0))},$ we have the blowup alternative, that is, for $\Pas \ \omega$, if $\tau_n(\omega)<\tau^*(u_0,v_0)(\omega),$ for every $n\in \mathbb{N}$, then
\begin{align*}
\lim_{t\to \tau^*(u_0,v_0)(\omega)}(||y(t)(\omega)||_{H^1}&+||z(t)(\omega)||_{H^1})=\infty.
\end{align*}
\end{theorem}
\begin{proof}
The proof is similar to the $L^2$-solutions case.
We solve \eqref{mild1} and \eqref{mild2}.
We consider \eqref{tauCM}.

First, we consider the case $d=1$. Then, the space $\mathcal{Y}_t$ is $(C\cap L^{\infty})([0,t];H^1)\cap L^4(0,t;W^{1,\infty})$. We estimate $\Phi(y,z)$ and $\Psi(y,z)$ as
\begin{align*}
||\Phi(y,z)||_{\mathcal{Y}_t}\lesssim& \ ||u_0||_{H^1}+t^{\frac{3}{4}}||z||_{L^{4}(0,t;W^{1,\infty})}||y||_{L^{\infty}(0,t;H^1)}, \\
||\Psi(y,z)||_{\mathcal{Y}_t}\lesssim& \ ||v_0||_{H^1}+t^{\frac{3}{4}}||y||_{L^{4}(0,t;W^{1,\infty})}||y||_{L^{\infty}(0,t;H^1)}.
\end{align*}
Similarly,
\begin{align*}
&||\Phi(y,z)-\Phi(y',z')||_{\mathcal{Y}_t} \\
\le& C_{t}t^{\frac{3}{4}}(||y||_{L^{4}(0,t;W^{1,\infty})}+||z'||_{L^{4}(0,t;W^{1,\infty})})(||y-y'||_{L^{\infty}(0,t;H^1)}+||z-z'||_{L^{\infty}(0,t;H^1)}), \\
&||\Psi(y,z)-\Psi(y',z')||_{\mathcal{Y}_t} \\
\le& C_{t}t^{\frac{3}{4}}(||y||_{L^{4}(0,t;W^{1,\infty})}+||y'||_{L^{4}(0,t;W^{1,\infty})})(||y-y'||_{L^{\infty}(0,t;H^1)}+||z-z'||_{L^{\infty}(0,t;H^1)}).
\end{align*}
Next, we consider the case $d=2$. Then, the space $\mathcal{Y}_t$ is $(C\cap L^{\infty})([0,t];H^1)\cap L^{q_0}(0,t;W^{1,r_0}),$ where $0<\frac{2}{q_0}=1-\frac{2}{r_0}<1$ with $r_0$ sufficiently large. We estimate $\Phi(y,z)$ and $\Psi(y,z)$ as
\begin{align*}
||\Phi(y,z)||_{\mathcal{Y}_t}\lesssim& \ ||u_0||_{H^1}+t^{\frac{r_0}{r_0+2}}||z||_{L^{r_0}(0,t;W^{1,\frac{2r_0}{r_0-2}})}||y||_{L^{\infty}(0,t;H^1)}, \\
||\Psi(y,z)||_{\mathcal{Y}_t}\lesssim& \ ||v_0||_{H^1}+t^{\frac{r_0}{r_0+2}}||y||_{L^{r_0}(0,t;W^{1,\frac{2r_0}{r_0-2}})}||y||_{L^{\infty}(0,t;H^1)}.
\end{align*}
Similarly,
\begin{align*}
&||\Phi(y,z)-\Phi(y',z')||_{\mathcal{Y}_t} \\
\le& C_{t}t^{\frac{r_0}{r_0+2}}(||y||_{L^{r_0}(0,t;W^{1,\frac{2r_0}{r_0-2}})}+||z'||_{L^{r_0}(0,t;W^{1,\frac{2r_0}{r_0-2}})})(||y-y'||_{L^{\infty}(0,t;H^1)}+||z-z'||_{L^{\infty}(0,t;H^1)}), \\
&||\Psi(y,z)-\Psi(y',z')||_{\mathcal{Y}_t} \\
\le& C_{t}t^{\frac{r_0}{r_0+2}}(||y||_{L^{r_0}(0,t;W^{1,\frac{2r_0}{r_0-2}})}+||y'||_{L^{r_0}(0,t;W^{1,\frac{2r_0}{r_0-2}})})(||y-y'||_{L^{\infty}(0,t;H^1)}+||z-z'||_{L^{\infty}(0,t;H^1)}).
\end{align*}
Note that
\begin{align*}
||y||_{L^{r_0}(0,t;W^{1,\frac{2r_0}{r_0-2}})}\le ||y||^{\frac{2}{r_0-2}}_{L^{q_0}(0,t;W^{1,r_0})}||y||^{\frac{r_0-4}{r_0-2}}_{L^{\infty}(0,t;H^1)}.
\end{align*}
Next, we consider the case $d=3$. Then, the space $\mathcal{Y}_t$ is $(C\cap L^{\infty})([0,t];H^1)\cap L^2(0,t;W^{1,6})$. We estimate $\Phi(y,z)$ and $\Psi(y,z)$ as
\begin{align*}
||\Phi(y,z)||_{\mathcal{Y}_t}\lesssim& ||u_0||_{H^1}+t^{\frac{1}{4}}||z||_{L^{2}(0,t;W^{1,6})}||y||_{L^{\infty}(0,t;H^1)}, \\
||\Psi(y,z)||_{\mathcal{Y}_t}\lesssim& \ ||v_0||_{H^1}+t^{\frac{1}{4}}||y||_{L^{2}(0,t;W^{1,6})}||y||_{L^{\infty}(0,t;H^1)}.
\end{align*}
Similarly,
\begin{align*}
&||\Phi(y,z)-\Phi(y',z')||_{\mathcal{Y}_t} \\
\le& C_{t}t^{\frac{1}{4}}(||y||_{L^{2}(0,t;W^{1,6})}+||z'||_{L^{2}(0,t;W^{1,6})})(||y-y'||_{L^{\infty}(0,t;H^1)}+||z-z'||_{L^{\infty}(0,t;H^1)}), \\
&||\Psi(y,z)-\Psi(y',z')||_{\mathcal{Y}_t} \\
\le& C_{t}t^{\frac{1}{4}}(||y||_{L^{2}(0,t;W^{1,6})}+||y'||_{L^{2}(0,t;W^{1,6})})(||y-y'||_{L^{\infty}(0,t;H^1)}+||z-z'||_{L^{\infty}(0,t;H^1)}).
\end{align*}
Next, we consider the case $d=4,5$. Then, the space $\mathcal{Y}_t$ is $(C\cap L^{\infty})([0,t];H^1)\cap L^2(0,t;W^{1,\frac{2d}{d-2}})$. Choose the Strichartz pair $(p,q)=(\frac{d}{2},\frac{4}{d-4})$ and the corresponding dual is given by $(p',q')=(\frac{d}{d-2},\frac{4}{8-d})$. We estimate $\Phi(y,z)$ as 
\begin{align*}
||\Phi(y,z)||_{\mathcal{Y}_t}\lesssim& ||u_0||_{H^1}+||z\overline{y}||_{L^{\frac{4}{8-d}}(0,t;W^{1,\frac{d}{d-2}})}.
\end{align*}
By the estimate
\begin{align*}
&||\nabla(z\overline{y})||_{L^{\frac{4}{8-d}}(0,t;L^{\frac{d}{d-2}})}\\
\lesssim& ||\nabla z||_{L^{\frac{4}{8-d}}(0,t;L^{\frac{2d}{d-2}})}||y||_{L^{\infty}(0,t;L^{\frac{2d}{d-2}})}+||z||_{L^{\frac{4}{8-d}}(0,t;L^{\frac{2d}{d-2}})}||\nabla y||_{L^{\infty}(0,t;L^{\frac{2d}{d-2}})}\\
\lesssim& t^{\frac{3}{2}-\frac{d}{4}}\left\{ ||\nabla z||_{L^{2}(0,t;L^{\frac{2d}{d-2}})}||y||_{L^{\infty}(0,t;L^{\frac{2d}{d-2}})}+||z||_{L^{2}(0,t;L^{\frac{2d}{d-2}})}||\nabla y||_{L^{\infty}(0,t;L^{\frac{2d}{d-2}})} \right\},
\end{align*}
we have
\begin{align*}
||\Phi(y,z)||_{\mathcal{Y}_t}\lesssim& ||u_0||_{H^1}+t^{\frac{3}{2}-\frac{d}{4}}||y||_{L^{\infty}(0,t;H^1)}||z||_{L^{2}(0,t;W^{1,\frac{2d}{d-2}})}.
\end{align*}
Note that $H^1\hookrightarrow L^{\frac{2d}{d-2}}$. We estimate $\Psi(y,z)$ as 
\begin{align*}
||\Psi(y,z)||_{\mathcal{Y}_t}\lesssim||v_0||_{H^1}+t^{\frac{3}{2}-\frac{d}{4}}||y||_{L^{\infty}(0,t;H^1)}||y||_{L^{2}(0,t;W^{1,\frac{2d}{d-2}})}.
\end{align*}
Similarly,
\begin{align*}
&||\Phi(y,z)-\Phi(y',z')||_{\mathcal{Y}_t} \\
\le& C_{t}||z\overline{y}-z'\overline{y'}||_{L^{\frac{4}{8-d}}(0,t;W^{1,\frac{d}{d-2}})} \\
\le& C_{t}t^{\frac{3}{2}-\frac{d}{4}} \left( ||z-z'||_{L^{\infty}(0,t;H^1)}||y||_{L^{2}(0,t;W^{1,\frac{2d}{d-2}})}+||z'||_{L^{2}(0,t;W^{1,\frac{2d}{d-2}})}||y-y'||_{L^{\infty}(0,t;H^1)} \right) \\
\le& C_{t}t^{\frac{3}{2}-\frac{d}{4}}(||y||_{L^{2}(0,t;W^{1,\frac{2d}{d-2}})}+||z'||_{L^{2}(0,t;W^{1,\frac{2d}{d-2}})})(||y-y'||_{L^{\infty}(0,t;H^1)}+||z-z'||_{L^{\infty}(0,t;H^1)}),
\end{align*}
\begin{align*}
&||\Psi(y,z)-\Psi(y',z')||_{\mathcal{Y}_t} \\
\le& C_{t}t^{\frac{3}{2}-\frac{d}{4}}(||y||_{L^{2}(0,t;W^{1,\frac{2d}{d-2}})}+||y'||_{L^{2}(0,t;W^{1,\frac{2d}{d-2}})})(||y-y'||_{L^{\infty}(0,t;H^1)}+||z-z'||_{L^{\infty}(0,t;H^1)}).
\end{align*}
Next, we consider the case $d=6$. Then, the space $\mathcal{Y}_t$ is $(C\cap L^{\infty})([0,t];H^1)\cap L^2(0,t;W^{1,3})$. We estimate $\Phi(y,z)$ and $\Psi(y,z)$ in $L^2(0,t;W^{1,3})$ as
\begin{align*}
||\Phi(y,z)||_{L^2(0,t;W^{1,3})}\lesssim& ||U_{\ell}(\cdot,\cdot)u_0||_{L^2(0,t;W^{1,3})}+||z\overline{y}||_{L^{1}(0,t;W^{1,2})} \\
\lesssim& ||U_{\ell}(\cdot,\cdot)u_0||_{L^2(0,t;W^{1,3})}+||z||_{L^{2}(0,t;L^6)}||y||_{L^{2}(0,t;W^{1,3})}+||z||_{L^{2}(0,t;W^{1,3})}||y||_{L^{2}(0,t;L^6)} \\
\lesssim& ||U_{\ell}(\cdot,\cdot)u_0||_{L^2(0,t;W^{1,3})}+||z||_{L^{2}(0,t;W^{1,3})}||y||_{L^{2}(0,t;W^{1,3})},
\end{align*}
\begin{align*}
||\Psi(y,z)||_{L^2(0,t;W^{1,3})}\lesssim||U_L(\cdot,\cdot)v_0||_{L^2(0,t;W^{1,3})}+||y||^2_{L^{2}(0,t;W^{1,3})}.
\end{align*}
Note that $W^{1,3}\hookrightarrow L^{6}$. Similarly,
\begin{align*}
&||\Phi(y,z)-\Phi(y',z')||_{L^2(0,t;W^{1,3})} \\
\le& C_{t}||z\overline{y}-z'\overline{y'}||_{L^{1}(0,t;W^{1,2})} \\
\le& C_{t} \left( ||z-z'||_{L^{2}(0,t;W^{1,3})}||y||_{L^{2}(0,t;W^{1,3})}+||z'||_{L^{2}(0,t;W^{1,3})}||y-y'||_{L^{2}(0,t;W^{1,3})} \right) \\
\le& C_{t}(||y||_{L^{2}(0,t;W^{1,3})}+||z'||_{L^{2}(0,t;W^{1,3})})(||y-y'||_{L^{2}(0,t;W^{1,3})}+||z-z'||_{L^{2}(0,t;W^{1,3})}),
\end{align*}
\begin{align*}
&||\Psi(y,z)-\Psi(y',z')||_{L^2(0,t;W^{1,3})} \\
\le& C_{t}(||y||_{L^{2}(0,t;W^{1,3})}+||y'||_{L^{2}(0,t;W^{1,3})})(||y-y'||_{L^{2}(0,t;W^{1,3})}+||z-z'||_{L^{2}(0,t;W^{1,3})}).
\end{align*}
For $u_0,v_0$, we know that $U_{\ell}(\cdot,\cdot)u_0,U_{L}(\cdot,\cdot)v_0\in L^2(0,t;W^{1,3})$, so by taking $T>0$ sufficiently small, the associated norm can be arbitrarily small. Thus, the argument of the contraction mapping works on a closed ball in $L^2(0,t;W^{1,3})$ centered at the origin and of a sufficiently small radius. Then the solution satisfies the integral equations $\mathcal{T}$ and belongs to $\mathcal{Y}_t$ by Strichartz estimates. \\\\
Finally, the blowup alternative follows in the same way as in the case of $L^2$-solutions.
\end{proof}

\section{Global well-posedness of $L^2$-solutions}

To show the global existence of $L^2$-solutions, we additionally assume
\begin{align}
\label{49}
\text{There exists a constant} \ c\in\mathbb{R}\backslash\{0\} \ \text{such that} \ \lambda=c\overline{\kappa}.
\end{align}
\begin{lem}
\label{lem1}
Assume $u_0,v_0\in L^2, \ 1\le d\le 4$ and $\lambda,\kappa$ satisfy (\ref{49}). Let $(u,v,(\tau_n)_{n\in\mathbb{N}},\tau^*(u_0,v_0))$ be the $L^2$-maximal solution of (\ref{SNLSS}). We have $\Pas$ for $0\le t<\tau^*(u_0,v_0)$
\begin{align}
\label{ucv}
Q(u,v)=&Q(u_0,v_0)+2\sum_{j=1}^N\int_0^t\int_{\mathbb{R}^d}\text{Re}(\mu_j)e_j|u(s)|^2d\xi d\beta_j(s) \nonumber \\
&+2c\sum_{j=1}^N\int_0^t\int_{\mathbb{R}^d}\text{Re}(\mu_j)e_j|v(s)|^2d\xi d\beta_j(s).
\end{align}
\end{lem}
\begin{proof}
Let $\{f_j\}_{j\ge1}\subset H^2(\mathbb{R}^d)$ be an orthonormal basis in $L^2(\mathbb{R}^d)$, set $J_{\varepsilon}=(I-\varepsilon \Delta)^{-1}$, $h_{\varepsilon}=J_{\varepsilon}h$ for any $h\in H^{-2}$ and $\phi_k=\mu_ke_k, \ 1\le k\le N$. From (\ref{14}) it follows that $\Pas$
\begin{align*}
u_{\varepsilon}(t)=&(u_0)_{\varepsilon}-\int_0^t[i\frac{1}{2\ell}\Delta u_{\varepsilon}(s)+(\mu u)_{\varepsilon}(s)+i\lambda (v\overline{u})_{\varepsilon}(s)]ds+\sum_{k=1}^N\int_0^t(u\phi_k)_{\varepsilon}(s)d\beta_k(s), \ t\in[0,\tau_n].
\end{align*}
Then for every $f_j, \ j\ge1$
\begin{align}
\label{massinner}
\langle f_j,u_{\varepsilon}(t)\rangle_{L^2}=&\langle f_j,(u_0)_{\varepsilon}\rangle_{L^2}-\langle f_j,\int_0^t[i\frac{1}{2\ell}\Delta u_{\varepsilon}(s)+(\mu u)_{\varepsilon}(s)+i\lambda (v\overline{u})_{\varepsilon}(s)]ds \rangle_{L^2} \nonumber \\
&+\langle f_j,\sum_{k=1}^N\int_0^t(u\phi_k)_{\varepsilon}(s)d\beta_k(s)\rangle_{L^2},  \ t\in[0,\tau_n].
\end{align}
By (\ref{uvn}), it can interchange the integrals for the drift term. While for the stochastic integral in (\ref{massinner}), we set
\[ \sigma_{n,m}=\inf\{t\in[0,\tau_n]:||u(t)||_{L^2}>m\}\wedge\tau_n. \]
Then, by the estimate $||J_{\varepsilon}f||_{H^k}\le ||f||_{H^k}$, we have
\begin{align*}
\mathbb{E}\int_0^{t\wedge\sigma_{n,m}}\sum_{k=1}^N|\langle f_j,(u\phi_k)_{\varepsilon}(s)\rangle_{L^2}|^2ds\le&\left( \sum_{k=1}^N||\phi_k||^2_{L^{\infty}} \right) ||f_j||^2_{L^{\infty}}\mathbb{E}\int_0^{t\wedge\sigma_{n,m}}||u(s)||^2_{L^2}ds \\
\le&\left( \sum_{k=1}^N||\phi_k||^2_{L^{\infty}} \right) ||f_j||^2_{L^{\infty}}m^2t<\infty.
\end{align*}
Therefore, by stochastic Fubini's theorem, we have
\begin{align}
\label{44}
\langle f_j,\sum_{k=1}^N\int_0^t(u\phi_k)_{\varepsilon}(s)d\beta_k(s)\rangle_{L^2}=\sum_{k=1}^N\int_0^t\langle f_j,(u\phi_k)_{\varepsilon}(s)\rangle_{L^2}d\beta_k(s),
\end{align}
holds on $\{ t\le\sigma_{n,m} \}$. By (\ref{uvn}), for $\Pas \ \omega\in\Omega$, there exists $m(\omega)\in\mathbb{N}$ such that $\sigma_{n,m}=\tau_n$ for $m\ge m(\omega)$. Hence
\begin{align}
\label{45}
\bigcup_{m\in\mathbb{N}}\{t \le \sigma_{n,m}\}=\{ t\le \tau_n\}.
\end{align}
This implies (\ref{44}) holds on $\{ t\le\tau_n\}$. So, we deduce that $\Pas$
\begin{align}
\langle f_j,u_{\varepsilon}(t)\rangle_{L^2}=&\langle f_j,(u_0)_{\varepsilon}\rangle_{L^2}-\int_0^t\langle f_j,i\frac{1}{2\ell}\Delta u_{\varepsilon}(s)+(\mu u)_{\varepsilon}(s)+i\lambda (v\overline{u})_{\varepsilon}(s)\rangle_{L^2}ds \nonumber \\
&+\sum_{k=1}^N\int_0^t\langle f_j,(u\phi_k)_{\varepsilon}(s)\rangle_{L^2}d\beta_k(s), \ t\in[0,\tau_n].
\end{align}
Applying the It\^{o} product rule, we obtain
\begin{align*}
&|\langle f_j,u_{\varepsilon}(t)\rangle_{L^2}|^2 \\
=&|\langle f_j,(u_0)_{\varepsilon}\rangle_{L^2}|^2+2\text{Re}\int_0^t\langle u_{\varepsilon}(s),f_j\rangle_{L^2}\langle f_j,-i\frac{1}{2\ell}\Delta u_{\varepsilon}(s)\rangle_{L^2} ds \\
&+2\text{Re}\int_0^t\langle u_{\varepsilon}(s),f_j\rangle_{L^2}\langle f_j,-(\mu u)_{\varepsilon}(s)\rangle_{L^2} ds+2\text{Re}\int_0^t\langle u_{\varepsilon}(s),f_j\rangle_{L^2}\langle f_j,-i\lambda (v\overline{u})_{\varepsilon}(s)\rangle_{L^2} ds \\
&+2\sum_{k=1}^N\text{Re}\int_0^t\langle u_{\varepsilon}(s),f_j\rangle_{L^2}\langle f_j,(u\phi_k)_{\varepsilon}(s)\rangle_{L^2} d\beta_k(s)+\sum_{k=1}^N\int_0^t|\langle f_j,(u\phi_k)_{\varepsilon}(s)\rangle_{L^2}|^2ds.
\end{align*}
Taking $j\in\mathbb{N}$ sums, interchanging infinite sums and integrals, it follows that
\begin{align}
\label{47}
||u_{\varepsilon}(t)||^2_{L^2}=&||(u_0)_{\varepsilon}||^2_{L^2}+2\text{Re}\int_0^t\langle u_{\varepsilon}(s),-i\frac{1}{2\ell}\Delta u_{\varepsilon}(s)\rangle_{L^2}ds \nonumber \\
&+2\text{Re}\int_0^t\langle u_{\varepsilon}(s),-(\mu u)_{\varepsilon}(s)\rangle_{L^2}ds+2\text{Re}\int_0^t\langle u_{\varepsilon}(s),-i\lambda(v\overline{u})_{\varepsilon}(s)\rangle_{L^2}ds \nonumber \\
&+2\text{Re}\sum_{k=1}^N\int_0^t\langle u_{\varepsilon}(s),(u\phi_k)_{\varepsilon}(s)\rangle_{L^2}d\beta_k(s)+\sum_{k=1}^N\int_0^t||(u\phi_k)_{\varepsilon}(s)||^2_{L^2}ds, \ t\in[0,\tau_n].
\end{align}
Similarly,
\begin{align}
\label{47v}
||v_{\varepsilon}(t)||^2_{L^2}=&||(v_0)_{\varepsilon}||^2_{L^2}+2\text{Re}\int_0^t\langle v_{\varepsilon}(s),-i\frac{1}{2L}\Delta v_{\varepsilon}(s)\rangle_{L^2}ds \nonumber \\
&+2\text{Re}\int_0^t\langle v_{\varepsilon}(s),-(\mu v)_{\varepsilon}(s)\rangle_{L^2}ds+2\text{Re}\int_0^t\langle v_{\varepsilon}(s),-i\kappa u^2_{\varepsilon}(s)\rangle_{L^2}ds \nonumber \\
&+2\text{Re}\sum_{k=1}^N\int_0^t\langle v_{\varepsilon}(s),(v\phi_k)_{\varepsilon}(s)\rangle_{L^2}d\beta_k(s)+\sum_{k=1}^N\int_0^t||(v\phi_k)_{\varepsilon}(s)||^2_{L^2}ds, \ t\in[0,\tau_n].
\end{align}
From \eqref{49}, combining \eqref{47} and \eqref{47v}, we have
\begin{align}
\label{47uv}
Q(u_{\varepsilon},v_{\varepsilon})=&Q(u_0,v_0)+2\text{Re}\int_0^t\langle u_{\varepsilon}(s),-i\frac{1}{2\ell}\Delta u_{\varepsilon}(s)\rangle_{L^2}ds+2c\text{Re}\int_0^t\langle v_{\varepsilon}(s),-i\frac{1}{2L}\Delta v_{\varepsilon}(s)\rangle_{L^2}ds \nonumber \\
&+2\text{Re}\int_0^t\langle u_{\varepsilon}(s),-(\mu u)_{\varepsilon}(s)\rangle_{L^2}ds+2c\text{Re}\int_0^t\langle v_{\varepsilon}(s),-(\mu v)_{\varepsilon}(s)\rangle_{L^2}ds \nonumber \\
&+2\text{Re}\int_0^t\langle u_{\varepsilon}(s),-i\lambda(v\overline{u})_{\varepsilon}(s)\rangle_{L^2}ds+2c\text{Re}\int_0^t\langle v_{\varepsilon}(s),-i\kappa u^2_{\varepsilon}(s)\rangle_{L^2}ds \nonumber \\
&+2\text{Re}\sum_{k=1}^N\int_0^t\langle u_{\varepsilon}(s),(u\phi_k)_{\varepsilon}(s)\rangle_{L^2}d\beta_k(s)+2c\text{Re}\sum_{k=1}^N\int_0^t\langle v_{\varepsilon}(s),(v\phi_k)_{\varepsilon}(s)\rangle_{L^2}d\beta_k(s) \nonumber \\
&+\sum_{k=1}^N\int_0^t||(u\phi_k)_{\varepsilon}(s)||^2_{L^2}ds+c\sum_{k=1}^N\int_0^t||(v\phi_k)_{\varepsilon}(s)||^2_{L^2}ds.
\end{align}
Here, the second term and the third term are zero. Therefore, since for $f\in L^p, \ p\in(1,\infty)$,
\begin{align*}
||J_{\varepsilon}f||_{L^p}&\le||f||_{L^p}, \\
J_{\varepsilon}f\to f \ &\text{in} \ L^p \ (\varepsilon\to0),
\end{align*}
we can pass to the limit $\varepsilon\to0$ in (\ref{47uv}). Take the stochastic integral (\ref{47uv}) as an example. Since as $\varepsilon\to0$,
\[ \langle u_{\varepsilon}(s),(u\phi_k)_{\varepsilon}(s)\rangle_{L^2}\to \langle u(s),u(s)\phi_k\rangle_{L^2}, \]
and
\begin{align*}
\mathbb{E}\int_0^{t\wedge\sigma_{n,m}}\sum_{k=1}^N(\text{Re}\langle u_{\varepsilon}(s),(u\phi_k)_{\varepsilon}(s)\rangle_{L^2})^2ds\le&\left( \sum_{k=1}^N||\phi_k||^2_{L^{\infty}} \right) \mathbb{E}\int_0^{t\wedge\sigma_{n,m}}||u(s)||^4_{L^2}ds \\
\le&\left( \sum_{k=1}^N||\phi_k||^2_{L^{\infty}} \right) m^4t<\infty,
\end{align*}
it follows that
\begin{align}
\label{48}
2\text{Re}\sum_{k=1}^N\int_0^t\langle u_{\varepsilon}(s),(u\phi_k)_{\varepsilon}(s)\rangle_{L^2}d\beta_k(s)\to 2\text{Re}\sum_{k=1}^N\int_0^t\langle u(s),u(s)\phi_k\rangle_{L^2}d\beta_k(s), 
\end{align}
in probability on $\{t\le \sigma_{n,m}\}$, which implies by (\ref{45}) that (\ref{48}) holds on $\{t\le \tau_n\}$.

After taking the limit of (\ref{47uv}), we notice that the fourth term and the fifth term cancel with the
tenth term and the
last term. As a result, we obtain $\Pas$ (\ref{ucv})
on $\{t\le\tau_n\}$, which implies that (\ref{ucv}) holds on $\{t\le\tau^*(u_0,v_0)\}$ as $\tau_n\to\tau^*(u_0,v_0), \ \Pas$
\end{proof}
\begin{lem}
\label{lem2}
Assume $u_0,v_0\in L^2, \ 1\le d\le 4$ and $\lambda,\kappa$ satisfy (\ref{49}). Let $(u,v,(\tau_n)_{n\in\mathbb{N}},\tau^*(u_0,v_0))$ be the $L^2$-maximal solution of (\ref{SNLSS}). Then, we have
\begin{align}
\label{410}
\mathbb{E} \sup_{0\le t<\tau^*(u_0,v_0)} (||u(t)||^2_{L^2}+c||v(t)||^2_{L^2})\le \widetilde{C}(T)<\infty.
\end{align}
\end{lem}
\begin{proof}
Taking into account that $\displaystyle \sum_{j=1}^N|\mu_j|||e_j||^2_{L^{\infty}}<\infty$ in \eqref{ucv}, by the Burkholder-Davis-Gundy and epsilon inequality, we have for $t\in[0,T]$ and $n\in\mathbb{N}$,
\begin{align*}
&\mathbb{E} \sup_{s\in[0,t\wedge\tau_n]}\left| \sum_{j=1}^N\int_0^s\int_{\mathbb{R}^d}\text{Re}(\mu_j)e_j|u(r)|^2d\xi d\beta_j(r)\right| \\
\le&\frac{1}{4}\mathbb{E}\sup_{s\in[0,t\wedge\tau_n]}||u(s)||^2_{L^2}+C\int_0^t\mathbb{E}\left(\sup_{r\in[0,s\wedge\tau_n]}||u(r)||^2_{L^2}\right)ds.
\end{align*}
Together with (\ref{ucv}), this yields
\begin{align}
\label{ucvest}
&\mathbb{E}\sup_{s\in[0,t\wedge\tau_n]}\left(||u(s)||^2_{L^2}+c||v(s)||^2_{L^2}\right) \nonumber \\
\le&2(||u_0||^2_{L^2}+c||v_0||^2_{L^2})+4C\int_0^t\mathbb{E}\sup_{r\in[0,s\wedge\tau_n]}\left(||u(r)||^2_{L^2}+c||v(r)||^2_{L^2}\right)ds
\end{align}
which implies
\[ \mathbb{E}\sup_{t\in[0,T\wedge\tau_n]}\left(||u(t)||^2_{L^2}+c||v(t)||^2_{L^2}\right)\le \widetilde{C}(T). \]
Finally, taking $n\uparrow \infty$ and applying Fatou's lemma, we obtain (\ref{410}).
\end{proof}
By Lemma \ref{lem1} and Lemma \ref{lem2}, we can proof the global existence of the $L^2$-solution for $1\le d\le 3$.
\begin{theorem}
\label{th41}
Assume (H1$)_0, \ 1\le d\le 3$ and $\lambda,\kappa$ satisfy (\ref{49}). Then, for each $u_0,v_0\in L^2$ and $0<T<\infty$, there exists a $L^2$-global solution $(u,v,T)$ of \textnormal{(\ref{SNLSS})}. It holds $\Pas$ that
\begin{align*}
u,v\in \mathcal{X}_{T},
\end{align*}
and uniqueness holds in the function space $\mathcal{X}_T$, where $\mathcal{X}_T$ is given by \eqref{solspX}.
\end{theorem}
\begin{proof}
Let $(u,v,(\tau_n)_{n\in\mathbb{N}},\tau^*(u_0,v_0))$ be the $L^2$-local solution of (\ref{SNLSS}). Then, by Lemma \ref{lem2}, we have 
\[ \mathbb{E} \sup_{0\le t<\tau^*(u_0,v_0)} (||u(t)||^2_{L^2}+||v(t)||^2_{L^2})\le \widetilde{C}(T)<\infty, \]
which implies
\[ \sup_{0\le t<\tau^*(u_0,v_0)} (||u(t)||^2_{L^2}+||v(t)||^2_{L^2})<\infty, \ \Pas \]
Then it follows from the blowup alternative in Theorem \ref{main} that $\tau^*(u_0,v_0)=T, \ \Pas$.
\end{proof}

\section{Global well-posedness of $H^1$-solutions}

We introduce the operators $\Theta_m, \ m\in\mathbb{N}$, used in \cite{BD03} and defined for any $f\in\mathcal{S}$ by
\[ \Theta_mf:=\mathcal{F}^{-1} \left( \theta\left(\frac{|\cdot|}{m}\right)\right)*f \quad (=m^d(\mathcal{F}^{-1}\theta)(m\cdot)*f). \]
where $\theta\in C^{\infty}_c$ is real-valued, nonnegative and $\theta(x)=1$  for $|x|\le 1, \ \theta(x)=0$ for $|x|>2$.

We consider the approximating equation
\begin{eqnarray}
\label{ASNLSS}
\begin{cases}
idu_m=\frac{1}{2\ell}\Delta u_mdt+\lambda \Theta_m(v_m\overline{u_m})dt-i\mu u_mdt+iu_mdW, \quad t\in(0,T), \ \xi\in \mathbb{R}^d, \\
idv_m=\frac{1}{2L}\Delta v_mdt+\kappa \Theta_m(u^2_m)dt-i\mu v_mdt+iv_mdW, \quad t\in(0,T), \ \xi\in \mathbb{R}^d, \\
u_m(0)=u_0, \quad v_m(0)=v_0, \quad \xi\in\mathbb{R}^d.
\end{cases}
\end{eqnarray}
\begin{lem}
\label{lem5}
Assume $u_0,v_0\in H^1$ and $1\le d\le 6$. Let $(u_m,v_m,(\tau_{n})_{n\in\mathbb{N}},\tau^*(u_0,v_0))$ be the $H^1$-maximal solution of (\ref{ASNLSS}). We have $\Pas$ for $0\le t<\tau^*(u_0,v_0)$
\begin{align}
\label{nab}
||\nabla u_m(t)||^2_{L^2}&=||\nabla u_0||^2_{L^2}+2\int_0^t\text{Re}\langle -\nabla(\mu u_m)(s),\nabla u_m(s)\rangle_{L^2}ds \nonumber \\
&+\sum_{j=1}^N\int_0^t||\nabla(u_m(s)\phi_j)||^2_{L^2}ds \nonumber \\
&-2\lambda \int_0^t\text{Re}\langle i\nabla \Theta_m(v_m\overline{u_m})(s),\nabla u_m(s)\rangle_{L^2}ds \nonumber \\
&+2\sum_{j=1}^N\int_0^t\text{Re}\langle\nabla (\phi_ju_m(s)),\nabla u_m(s)\rangle_{L^2}d\beta_j(s), \\
||\nabla v_m(t)||^2_{L^2}&=||\nabla v_0||^2_{L^2}+2\int_0^t\text{Re}\langle -\nabla(\mu v_m)(s),\nabla v_m(s)\rangle_{L^2}ds \nonumber \\
&+\sum_{j=1}^N\int_0^t||\nabla(v_m(s)\phi_j)||^2_{L^2}ds \nonumber \\
&-2\kappa \int_0^t\text{Re}\langle i\nabla \Theta_m(u^2_m)(s),\nabla v_m(s)\rangle_{L^2}ds \nonumber \\
&+2\sum_{j=1}^N\int_0^t\text{Re}\langle\nabla (\phi_jv_m(s)),\nabla v_m(s)\rangle_{L^2}d\beta_j(s).
\end{align}
\end{lem}
\begin{proof}
We only prove (\ref{nab}). Let $\{f_j\}_{j\ge1}\subset H^2(\mathbb{R}^d)$ be an orthonormal basis in $L^2(\mathbb{R}^d)$, set $J_{\varepsilon}=(I-\varepsilon \Delta)^{-1}$, $h_{\varepsilon}=J_{\varepsilon}h$ for any $h\in H^{-1}$ and $\phi_k=\mu_ke_k, \ 1\le k\le N$. Then, we have $\Pas$ for $t\in(0,\tau^*(u_0,v_0))$
\begin{eqnarray}
\label{idu}
\begin{cases}
\displaystyle
idu_{m,\varepsilon}=\frac{1}{2\ell}\Delta u_{m,\varepsilon}dt-i(\mu u_m)_{\varepsilon}dt+\lambda g_{m,\varepsilon}dt+\sum_{j=1}^Ni(u_m\phi_j)_{\varepsilon}d\beta_j, \\
u_{m,\varepsilon}(0)=u_{0,\varepsilon},
\end{cases}
\end{eqnarray}
where $g_{m,\varepsilon}=[\Theta_m(v_m\overline{u_m})]_{\varepsilon}$.

Noticing that $\partial_{l}f_k\in H^1, \ 1\le l\le d$ for each $f_k,k\in\mathbb{N}$, it follows from (\ref{idu}) and Fubini's theorem that $\Pas$ for $t\in(0,\tau^*(u_0,v_0))$
\begin{align*}
\langle u_{m,\varepsilon}(t),\partial_lf_k\rangle_{L^2}=&\langle u_{0,\varepsilon},\partial_lf_k\rangle_{L^2}+\int_0^t\langle -i\frac{1}{2\ell}\Delta u_{m,\varepsilon}(s),\partial_lf_k\rangle_{L^2}ds+\int_0^t\langle -(\mu u_m)_{\varepsilon}(s),\partial_lf_k\rangle_{L^2}ds \\
&+\int_0^t\langle -\lambda ig_{m,\varepsilon}(s),\partial_lf_k\rangle_{L^2}ds+\sum_{j=1}^N\int_0^t\langle (u_m(s)\phi_j)_{\varepsilon},\partial_lf_k\rangle_{L^2}d\beta_j(s).
\end{align*}
Applying It\^{o}'s product rule and integrating by parts, we obtain
\begin{align*}
|\langle u_{m,\varepsilon}(t),\partial_lf_k\rangle_{L^2}|^2=&|\langle u_{0,\varepsilon},\partial_lf_k\rangle_{L^2}|^2 \\
&+2\text{Re}\int_0^t\overline{\langle \partial_lu_{m,\varepsilon}(s),f_k\rangle_{L^2}}\langle -i\frac{1}{2\ell}\partial_l\Delta u_{m,\varepsilon}(s),f_k\rangle_{L^2}ds \\
&+2\text{Re}\int_0^t\overline{\langle \partial_lu_{m,\varepsilon}(s),f_k\rangle_{L^2}}\langle -\partial_l(\mu u_m)_{\varepsilon}(s),f_k\rangle_{L^2}ds \\
&+2\text{Re}\int_0^t\overline{\langle \partial_lu_{m,\varepsilon}(s),f_k\rangle_{L^2}}\langle -\lambda i\partial_lg_{m,\varepsilon}(s),f_k\rangle_{L^2}ds \\
&+2\sum_{j=1}^N\text{Re}\int_0^t\overline{\langle \partial_lu_{m,\varepsilon}(s),f_k\rangle_{L^2}}\langle \partial_l(u_m(s)\phi_j)_{\varepsilon},f_k\rangle_{L^2}d\beta_j(s) \\
&+\sum_{j=1}^N\int_0^t|\langle \partial_l(u_m(s)\phi_j)_{\varepsilon},f_k\rangle_{L^2}|^2ds, \quad t<\tau^*(u_0,v_0), \ \Pas
\end{align*}
Taking $k\in\mathbb{N}$ sums, interchanging infinite sums and integrals, we obtain $\Pas$ for all $t\in (0,\tau^*(u_0,v_0))$
\begin{align*}
&||\partial_lu_{m,\varepsilon}(t)||_{L^2}^2=\sum_{k=1}^{\infty}|\langle u_{m,\varepsilon}(t),\partial_lf_k\rangle_{L^2}|^2 \\
=&||\partial_lu_{0,\varepsilon}||^2_{L^2}+2\int_0^t\text{Re}\langle i\frac{1}{2\ell}\Delta u_{m,\varepsilon}(s),\partial_l^2u_{m,\varepsilon}(s)\rangle_{L^2}ds \\
&+2\int_0^t\text{Re}\langle -\partial_l(\mu u_m)_{\varepsilon}(s),\partial_lu_{m,\varepsilon}(s)\rangle_{L^2}ds+\sum_{j=1}^N\int_0^t||\partial_l(u_m(s)\phi_j)_{\varepsilon}||^2_{L^2}ds \\
&-2\lambda\int_0^t\text{Re}\langle i\partial_lg_{m,\varepsilon}(s),\partial_lu_{m,\varepsilon}(s)\rangle_{L^2}ds+2\sum_{j=1}^N\int_0^t\text{Re}\langle \partial_l(u_m(s)\phi_j)_{\varepsilon},\partial_lu_{m,\varepsilon}(s)\rangle_{L^2}d\beta_j(s).
\end{align*}
Finally, summing over $1\le l\le d,$ and by
\begin{align*}
J_{\varepsilon}f&\to f \ \text{in} \ H^k, \\
||J_{\varepsilon}f||_{H^k}&\le||f||_{H^k}, \ k=-1,0,1,
\end{align*}
we can pass to the limit $\varepsilon\to0$ in the above equality and conclude the formula (\ref{nab}).
\end{proof}
\begin{theorem}
\label{th52}
Assume $u_0,v_0\in H^1,\ 1\le d\le 6$ and $\lambda,\kappa$ satisfy (\ref{49}). Let $(u,v,(\tau_{n})_{n\in\mathbb{N}},\tau^*(u_0,v_0))$ be the $H^1$-maximal solution of (\ref{SNLSS}). We have $\Pas$ for $0\le t<\tau^*(u_0,v_0)$
\begin{align}
\label{58re}
E(u,v)=& \ E(u_0,v_0)+\frac{1}{\ell}\int_0^t\text{Re}\langle -\nabla(\mu u)(s),\nabla u(s)\rangle_{L^2}ds+\frac{1}{2\ell}\sum_{j=1}^N\int_0^t||\nabla(u(s)\phi_j)||^2_{L^2}ds \nonumber \\
&+\frac{c}{2L}\int_0^t\text{Re}\langle -\nabla(\mu v)(s),\nabla v(s)\rangle_{L^2}ds+\frac{c}{4L}\sum_{j=1}^N\int_0^t||\nabla(v(s)\phi_j)||^2_{L^2}ds \nonumber \\
&+3\int_0^t\text{Re}\langle \lambda v,\mu u^2\rangle_{L^2}ds-3\sum_{j=1}^N\int_0^t\text{Re}(\overline{\phi_j}\langle \lambda v,u^2\rangle_{L^2})d\beta_j(s) \nonumber \\ 
&+\frac{1}{\ell
}\sum_{j=1}^N\text{Re}\int_0^t\langle\nabla (\phi_ju(s)),\nabla u(s)\rangle_{L^2}d\beta_j(s)+\frac{c}{2L}\sum_{j=1}^N\int_0^t\text{Re}\langle\nabla (\phi_jv(s)),\nabla v(s)\rangle_{L^2}d\beta_j(s).
\end{align}
\end{theorem}
\begin{proof}
By Lemma \ref{lem5}, we have
\begin{align}
\label{umvm}
&\frac{1}{2\ell}||\nabla u_m(t)||^2_{L^2}+\frac{c}{4L}||\nabla v_m(t)||^2_{L^2} \nonumber \\
=&\frac{1}{2\ell}||\nabla u_0||^2_{L^2}+\frac{c}{4L}||\nabla v_0||^2_{L^2} \nonumber \\
&+\frac{1}{\ell}\int_0^t\text{Re}\langle -\nabla(\mu u_m)(s),\nabla u_m(s)\rangle_{L^2}ds+\frac{c}{2L}\int_0^t\text{Re}\langle -\nabla(\mu v_m)(s),\nabla v_m(s)\rangle_{L^2}ds \nonumber \\
&+\frac{1}{2\ell}\sum_{j=1}^N\int_0^t||\nabla(u_m(s)\phi_j)||^2_{L^2}ds+\frac{c}{4L}\sum_{j=1}^N\int_0^t||\nabla(v_m(s)\phi_j)||^2_{L^2}ds \nonumber \\
&-\frac{\lambda}{\ell} \int_0^t\text{Re}\langle i\nabla \Theta_m(v_m\overline{u_m})(s),\nabla u_m(s)\rangle_{L^2}ds-\frac{c\kappa}{2L} \int_0^t\text{Re}\langle i\nabla \Theta_m(u^2_m)(s),\nabla v_m(s)\rangle_{L^2}ds \nonumber \\
&+\frac{1}{\ell}\sum_{j=1}^N\text{Re}\int_0^t\langle\nabla (\phi_ju_m(s)),\nabla u_m(s)\rangle_{L^2}d\beta_j(s)+\frac{c}{2L}\sum_{j=1}^N\int_0^t\text{Re}\langle\nabla (\phi_jv_m(s)),\nabla v_m(s)\rangle_{L^2}d\beta_j(s).
\end{align}
First, we calculate
\begin{align}
\label{reum}
&-\frac{\lambda}{\ell} \int_0^t\text{Re}\langle i\nabla \Theta_m(v_m\overline{u_m})(s),\nabla u_m(s)\rangle_{L^2}ds=-\frac{\lambda}{\ell}\text{Re}\int\int_0^t\overline{i}\Theta_m(v_m\overline{u_m})\overline{\Delta u_m(s)}dsd\xi \nonumber \\
=& \ -\frac{\lambda}{\ell} \text{Re}\int\int_0^t\overline{i}\Theta_m(v_m\overline{u_m})\left\{ \overline{2\ell idu_m}-\overline{2\ell \lambda \Theta_m(v_m\overline{u_m})}dt+\overline{2\ell i\mu u_m}dt-\overline{2\ell iu_mdW}\right\} d\xi, \nonumber \\
=& \ \int_0^t\text{Re}\langle \lambda v_m,\partial_t\Theta_m(u^2_m)\rangle_{L^2}ds+2\int_0^t\text{Re}\langle \lambda \Theta_m(v_m\overline{u_m}),i\lambda \Theta_m(v_m\overline{u_m})\rangle_{L^2}ds \nonumber \\
&+2\int_0^t\text{Re}\langle \lambda v_m,\mu \Theta_m(u^2_m)\rangle_{L^2}ds-2\sum_{j=1}^N\text{Re}\int_0^t\overline{\phi_j}\langle \lambda v_m,\Theta_m(u^2_m)\rangle_{L^2}d\beta_j(s), \\
\label{revm}
&-\frac{c\kappa}{2L} \int_0^t\text{Re}\langle i\nabla \Theta_m(u^2_m)(s),\nabla v_m(s)\rangle_{L^2}ds=-\frac{c\kappa}{2L}\text{Re}\int\int_0^t\overline{i}\Theta_m(u^2_m)\overline{\Delta v_m(s)}dsd\xi \nonumber \\
=& \ -\frac{c\kappa}{2L} \text{Re}\int\int_0^t\overline{i}\Theta_m(u^2_m)\left\{ \overline{2Lidv_m}-\overline{2L \kappa \Theta_m(u^2_m)}dt+\overline{2Li\mu v_m}dt-\overline{2Liv_mdW}\right\} d\xi \nonumber \\
=& c\int_0^t\text{Re}\langle \kappa \Theta_m(u^2_m),\partial_tv_m\rangle_{L^2}ds+c\int_0^t\text{Re}\langle \kappa \Theta_m(u^2_m),i\kappa \Theta_m(u^2_m)\rangle_{L^2}ds \nonumber \\
&+c\int_0^t\text{Re}\langle \kappa \Theta_m(u^2_m),\mu v_m\rangle_{L^2}ds-c\sum_{j=1}^N\text{Re}\int_0^t\overline{\phi_j}\langle \kappa v_m,\Theta_m(u^2_m)\rangle_{L^2}d\beta_j(s),
\end{align}
where the second term on the right-hand side is zero, respectively.
We consider the first term on the right side of \eqref{reum} and \eqref{revm}. By \eqref{49}, we have
\begin{align*}
&\text{Re}\int_0^t\langle \lambda v_m,\partial_t\Theta_m(u^2_m)\rangle_{L^2}ds+c\text{Re}\int_0^t\langle \kappa \Theta_m(u^2_m),\partial_tv_m\rangle_{L^2}ds \\
=&\text{Re}\int_0^t(\langle \lambda v_m,\partial_t\Theta_m(u^2_m)\rangle_{L^2}+\overline{\langle \lambda\partial_tv_m,\Theta_m(u^2_m)\rangle_{L^2}})ds=\text{Re}\int_0^t\frac{\partial}{\partial t}\lambda \langle v_m,\Theta_m(u^2_m)\rangle_{L^2}ds \\
=&\text{Re}(\lambda \langle v_m(t),\Theta_m(u^2_m)(t)\rangle_{L^2})-\text{Re}(\lambda\langle v_0,u^2_0\rangle_{L^2}).
\end{align*}
Taking the limit $m\to\infty$ in \eqref{umvm}, we have
\begin{align*}
E(u,v)=& \ E(u_0,v_0)+\frac{1}{\ell}\int_0^t\text{Re}\langle -\nabla(\mu u)(s),\nabla u(s)\rangle_{L^2}ds+\frac{c}{2L}\int_0^t\text{Re}\langle -\nabla(\mu v)(s),\nabla v(s)\rangle_{L^2}ds \nonumber \\
&+\frac{1}{2\ell}\sum_{j=1}^N\int_0^t||\nabla(u(s)\phi_j)||^2_{L^2}ds+\frac{c}{4L}\sum_{j=1}^N\int_0^t||\nabla(v(s)\phi_j)||^2_{L^2}ds \nonumber \\
&+2\int_0^t\text{Re}\langle \lambda v,\mu u^2\rangle_{L^2}ds-2\sum_{j=1}^N\text{Re}\int_0^t\overline{\phi_j}\langle \lambda v,u^2\rangle_{L^2}d\beta_j(s), \\
&+c\int_0^t\text{Re}\langle \kappa u^2,\mu v\rangle_{L^2}ds-c\sum_{j=1}^N\text{Re}\int_0^t\overline{\phi_j}\langle \kappa v,u^2\rangle_{L^2}d\beta_j(s), \\
&+\frac{1}{\ell}\sum_{j=1}^N\text{Re}\int_0^t\langle\nabla (\phi_ju(s)),\nabla u(s)\rangle_{L^2}d\beta_j(s)+\frac{c}{2L}\sum_{j=1}^N\int_0^t\text{Re}\langle\nabla (\phi_jv(s)),\nabla v(s)\rangle_{L^2}d\beta_j(s).
\end{align*}
where we note that 
\[ u_m\to u, \quad \nabla\Theta_m(v_m\overline{u_m})\to \nabla (v\overline{u}). \]
Next, we consider the sixth term and the eighth term on the right side. We have
\begin{align*}
2\text{Re}\int_0^t\langle \lambda v,\mu u^2\rangle_{L^2}ds+c\text{Re}\int_0^t\langle \kappa u^2,\mu v\rangle_{L^2}ds=3\int_0^t\text{Re}\langle \lambda v,\mu u^2 \rangle_{L^2}ds.
\end{align*}
Finally, we consider the seventh term and the ninth term on the right side. We have
\begin{align*}
&-2\sum_{j=1}^N\text{Re}\int_0^t\overline{\phi_j}\langle \lambda v,u^2\rangle_{L^2}d\beta_j(s)-c\sum_{j=1}^N\text{Re}\int_0^t\overline{\phi_j}\langle \kappa u^2,v\rangle_{L^2}d\beta_j(s) \\
=&-3\sum_{j=1}^N\int_0^t\text{Re}(\overline{\phi_j}\langle \lambda v,u^2\rangle_{L^2})d\beta_j(s).
\end{align*}
This completes the proof.
\end{proof}
\begin{lem}
\label{lem54}
(\cite[Lemma 3.4]{BRZ16})
Let $Y\ge0$ be real-valued progressively measurable process, we have
\begin{align}
\mathbb{E}\left( \int_0^tY^2(s)ds\right)^{\frac{1}{2}}\le\varepsilon\mathbb{E}\sup_{s\le t}Y(s)+C_{\varepsilon}\int_0^t\mathbb{E}\sup_{r\le s}Y(r)ds.
\end{align}
\end{lem}
To show the global existence of $H^1$-solutions, we additionally assume
\begin{align}
\label{H1add}
\text{Re}\mu_j=0 \ \text{or} \ \text{``} \ c>0 \ \text{and} \ \text{Re}(\mu_j)e_j\le 0, \ \text{for each} \ j \ \text{''}.
\end{align}
\begin{lem}
\label{lem56}
Assume $1\le d\le3, \ \lambda,\kappa$ satisfy (\ref{49}) and \eqref{H1add}. Then, we have
\begin{align}
\label{sup2}
\mathbb{E}\sup_{t\in[0,\tau^*(u_0,v_0))}(\frac{1}{2\ell}||\nabla u(t)||^2_{L^2}+\frac{c}{4L}||\nabla v(t)||^2_{L^2})\le \widetilde{C}(T)<\infty.
\end{align}
\end{lem}
\begin{proof}
By Theorem \ref{th52}, for all $n\ge1$ and $t\in[0,T]$, we have $\Pas$
\begin{align*}
&\frac{1}{2\ell}||\nabla u(t\wedge\tau_n)||^2_{L^2}+\frac{c}{4L}||\nabla v(t\wedge\tau_n)||^2_{L^2} \\
=&\frac{1}{2\ell}||\nabla u_0||^2_{L^2}+\frac{c}{4L}||\nabla v_0||^2_{L^2}+\text{Re}(\lambda\langle v(t\wedge\tau_n),u^2(t\wedge\tau_n)\rangle_{L^2})-\text{Re}(\lambda\langle v_0,u_0^2\rangle_{L^2}) \\
&+\frac{1}{\ell}\int_0^{t\wedge\tau_n}\left[ \text{Re}\langle -\nabla(\mu u)(s),\nabla u(s)\rangle_{L^2}+\frac{1}{2}\sum_{j=1}^N||\nabla(u(s)\phi_j)||^2_{L^2}\right] ds \\
&+\frac{c}{2L}\int_0^{t\wedge\tau_n}\left[ \text{Re}\langle -\nabla(\mu v)(s),\nabla v(s)\rangle_{L^2}+\frac{1}{2}\sum_{j=1}^N||\nabla(v(s)\phi_j)||^2_{L^2}\right] ds \\
&+3\int_0^{t\wedge\tau_n}\text{Re}\langle \lambda v,\mu u^2\rangle_{L^2}ds-3\sum_{j=1}^N\int_0^{t\wedge\tau_n}\text{Re}(\overline{\phi_j}\langle \lambda v,u^2\rangle_{L^2})d\beta_j(s) \\ 
&+\frac{1}{\ell
}\sum_{j=1}^N\int_0^{t\wedge\tau_n}\text{Re}\langle\nabla (\phi_ju(s)),\nabla u(s)\rangle_{L^2}d\beta_j(s) \\
&+\frac{c}{2L}\sum_{j=1}^N\int_0^{t\wedge\tau_n}\text{Re}\langle\nabla (\phi_jv(s)),\nabla v(s)\rangle_{L^2}d\beta_j(s) \\
=:&\frac{1}{2\ell}||\nabla u_0||^2_{L^2}+\frac{c}{4L}||\nabla v_0||^2_{L^2}+\text{Re}(\lambda\langle v(t\wedge\tau_n),u^2(t\wedge\tau_n)\rangle_{L^2})-\text{Re}(\lambda\langle v_0,u_0^2\rangle_{L^2}) \\
&+\sum_{k=1}^6J_k(t\wedge\tau_n).
\end{align*}
Estimate the third term on the right-hand side.
By the Gagliardo-Nirenberg inequality
\[ ||\phi||_{L^3}\le C||\nabla \phi||^{\frac{d}{6}}_{L^2}||\phi||^{1-\frac{d}{6}}_{L^2}, \]
we estimate
\begin{align}
\label{inter}
|\lambda\langle v,u^2\rangle_{L^2}|\le&|\lambda|C^3||\nabla u||^{\frac{d}{3}}_{L^2}||u||^{2-\frac{d}{3}}_{L^2}||\nabla v||^{\frac{d}{6}}_{L^2}||v||^{1-\frac{d}{6}}_{L^2} \nonumber \\
\le&|\lambda|C^3(2\ell K(u,v))^{\frac{d}{6}}Q(u,v)^{1-\frac{d}{6}}\left( \frac{4L}{c}K(u,v)\right)^{\frac{d}{12}}\left( \frac{1}{c}Q(u,v)\right)^{\frac{1}{2}-\frac{d}{12}} \nonumber \\
=&|\lambda|C^3(16\ell^2 L)^{\frac{d}{12}}c^{-\frac{1}{2}}K(u,v)^{\frac{d}{4}}Q(u,v)^{\frac{3}{2}-\frac{d}{4}} \nonumber \\
\le&|\lambda|C^3(16\ell^2 L)^{\frac{d}{12}}c^{-\frac{1}{2}}(\varepsilon K(u,v)+C_{\varepsilon}Q(u,v)^{\frac{6-d}{4-d}}) \nonumber \\
=&|\lambda|C^3(16\ell^2 L)^{\frac{d}{12}}c^{-\frac{1}{2}}(\varepsilon K(u,v)+C_{\varepsilon}Q(u_0,v_0)^{\frac{6-d}{4-d}}). 
\end{align}
Here, the last equality using assumptions \eqref{H1add} and \eqref{ucv}. Therefore,
\begin{align}
\label{eq11}
\mathbb{E}\sup_{s\le t\wedge\tau_n}\text{Re}(\lambda\langle v(s),u^2(s)\rangle_{L^2}) \le&|\lambda|\mathbb{E}\sup_{s\le t\wedge\tau_n}|\langle v(s),u^2(s)\rangle_{L^2}| \nonumber \\
\le&C\mathbb{E}\sup_{s\le t\wedge\tau_n}(\varepsilon K(u,v)+C_{\varepsilon}Q(u_0,v_0)^{\frac{6-d}{4-d}}) \nonumber \\
\le&C\varepsilon\mathbb{E}\sup_{s\le t\wedge\tau_n}(\frac{1}{2\ell}||\nabla u(s)||^2_{L^2}+\frac{c}{4L}||\nabla v(s)||^2_{L^2})+\widetilde{C}(T).
\end{align}
If $J_1$ takes $\varepsilon$ sufficiently small from the Lemma \ref{lem2}, then the following holds.
\begin{align}
\mathbb{E}\sup_{s\le t\wedge\tau_n}J_1(s)\le&\frac{C}{\ell}\mathbb{E}\sup_{s\le t\wedge\tau_n}\int_0^s\{ ||\nabla u(r)||^2_{L^2}+||u(r)||^2_{L^2} \} dr \nonumber \\
\le&\frac{C}{\ell}\widetilde{C}(T)t+\frac{C}{\ell}\int_0^t\mathbb{E}\sup_{r\le s\wedge\tau_n}||\nabla u(r)||^2_{L^2}ds. 
\end{align}
Similarly,
\begin{align}
\mathbb{E}\sup_{s\le t\wedge\tau_n}J_2(s)\le\frac{cC}{2L}\widetilde{C}(T)t+\frac{cC}{2L}\int_0^t\mathbb{E}\sup_{r\le s\wedge\tau_n}||\nabla v(r)||^2_{L^2}ds. 
\end{align}
By (\ref{inter}), $J_3$ is estimated as follows.
\begin{align}
\mathbb{E}\sup_{s\le t\wedge\tau_n}|J_3(s)| \le&C\mathbb{E}\sup_{s\le t\wedge\tau_n}\int_0^s(\varepsilon K(u(r),v(r))+C_{\varepsilon}Q(u_0,v_0)^{\frac{6-d}{4-d}})dr \nonumber \\
\le&C\int_0^t\mathbb{E}\sup_{r\le s\wedge\tau_n}(\frac{1}{2\ell}||\nabla u(r)||^2_{L^2}+\frac{c}{4L}||\nabla v(r)||^2_{L^2})ds+\widetilde{C}(T).
\end{align}
To estimate $J_4$, by (\ref{inter}), Burkholder-Davis-Gundy inequality and Lemma \ref{lem54}, we have
\begin{align}
\mathbb{E}\sup_{s\le t\wedge\tau_n}|J_4(s)|\le&C\sum_{j=1}^N\mathbb{E}\sup_{s\le t\wedge\tau_n}\int_0^s(\varepsilon K(u(r),v(r))+C_{\varepsilon}Q(u_0,v_0)^{\frac{6-d}{4-d}})d\beta_j(r) \nonumber \\
\le&C\mathbb{E}\left[ \int_0^{t\wedge\tau_n}(\varepsilon K(u(s),v(s))+C_{\varepsilon}Q(u_0,v_0)^{\frac{6-d}{4-d}})^2ds \right]^{\frac{1}{2}} \nonumber \\
\le&C\varepsilon'\mathbb{E}\sup_{s\le t}(\varepsilon K(u(s),v(s))+C_{\varepsilon}Q(u_0,v_0)^{\frac{6-d}{4-d}}) \nonumber \\
&+C_{\varepsilon'}\int_0^t\mathbb{E}\sup_{r\le s}(\varepsilon K(u(r),v(r))+C_{\varepsilon}Q(u_0,v_0)^{\frac{6-d}{4-d}})ds \nonumber \\
\le&C\varepsilon\mathbb{E}\sup_{s\le t}(\frac{1}{2\ell}||\nabla u(s)||^2_{L^2}+\frac{c}{4L}||\nabla v(s)||^2_{L^2}) \nonumber \\
&+C_{\varepsilon}\int_0^t\mathbb{E}\sup_{r\le s}(\frac{1}{2\ell}||\nabla u(r)||^2_{L^2}+\frac{c}{4L}||\nabla v(r)||^2_{L^2})ds+\widetilde{C}(T).
\end{align}
To estimate $J_5$, by Burkholder-Davis-Gundy inequality, we have
\begin{align*}
\mathbb{E}\sup_{s\le t\wedge\tau_n}|J_5(s)|\le&\frac{C}{\ell}\mathbb{E}\left[ \int_0^{t\wedge\tau_n}(\text{Re}\langle \nabla(\phi_ju(s)),\nabla u(s)\rangle_{L^2})^2ds \right]^{\frac{1}{2}} \nonumber \\
\le&\frac{C}{\ell}\mathbb{E}\left[ \int_0^{t\wedge\tau_n}(||u(s)||^4_{L^2}+||\nabla u(s)||^4_{L^2})ds \right]^{\frac{1}{2}} \\
\le&\frac{C}{\ell}\mathbb{E}\left(\int_0^{t\wedge\tau_n}||u(s)||^4_{L^2}ds\right)^{\frac{1}{2}}+\frac{C}{\ell}\mathbb{E}\left(\int_0^{t\wedge\tau_n}||\nabla u(s)||^4_{L^2}ds\right)^{\frac{1}{2}}.
\end{align*}
Then, if we replace $Y$ in Lemma \ref{lem54} by $||u(s)||^2_{L^2}$ and $||\nabla u(s)||^2_{L^2}$, and using Lemma \ref{lem2} we have
\begin{align}
\mathbb{E}\sup_{s\le t\wedge\tau_n}|J_5(s)|\le&\varepsilon \frac{C}{\ell}\widetilde{C}(T)+\frac{C}{\ell}C_{\varepsilon}\widetilde{C}(T)t+\varepsilon \frac{C}{\ell}\mathbb{E}\sup_{s\le t\wedge\tau_n}||\nabla u(s)||^2_{L^2} \nonumber \\
&+\frac{C}{\ell}C_{\varepsilon}\int_0^t\mathbb{E}\sup_{r\le s\wedge\tau_n}||\nabla u(r)||^2_{L^2}ds.
\end{align}
Similarly,
\begin{align}
\label{eq7}
\mathbb{E}\sup_{s\le t\wedge\tau_n}|J_6(s)|\le&\varepsilon \frac{cC}{2L}\widetilde{C}(T)+\frac{cC}{2L}C_{\varepsilon}\widetilde{C}(T)t+\varepsilon \frac{cC}{2L}\mathbb{E}\sup_{s\le t\wedge\tau_n}||\nabla v(s)||^2_{L^2} \nonumber \\
&+\frac{cC}{2L}C_{\varepsilon}\int_0^t\mathbb{E}\sup_{r\le s\wedge\tau_n}||\nabla v(r)||^2_{L^2}ds.
\end{align}
Thus, summarizing (\ref{eq11})-(\ref{eq7}), we obtain
\begin{align*}
&\mathbb{E}\sup_{s\le t\wedge\tau_n}\left( \frac{1}{2\ell}||\nabla u(s)||^2_{L^2}+\frac{c}{4L}||\nabla v(s)||^2_{L^2} \right) \nonumber \\
\le&C_1(T)+\varepsilon C_2(T)\mathbb{E}\sup_{s\le t\wedge\tau_n}\left(\frac{1}{2\ell}||\nabla u(s)||^2_{L^2}+\frac{c}{4L}||\nabla v(s)||^2_{L^2}\right) \nonumber \\
&+C_3(T)\int_0^t\mathbb{E}\sup_{r\le s\wedge\tau_n}\left(\frac{1}{2\ell}||\nabla u(r)||^2_{L^2}+\frac{c}{4L}||\nabla v(r)||^2_{L^2} \right)ds. 
\end{align*}
Thus, if we take $\varepsilon$ small enough, from Gronwall's lemma, we have
\[ \mathbb{E}\sup_{t\in[0,\tau_n]}\left(\frac{1}{2\ell}||\nabla u(t)||^2_{L^2}+\frac{c}{4L}||\nabla v(t)||^2_{L^2} \right)\le \widetilde{C}(T)<\infty. \]
Taking $n\to \infty$ and applying Fatou's lemma, we obtain (\ref{sup2}), as
claimed.
\end{proof}
\begin{theorem}
\label{th53}
Assume (H1$)_1$, $1\le d\le3, \ \lambda,\kappa$ satisfy \eqref{49} and \eqref{H1add}. Then, for each $u_0,v_0\in H^1$ and $0<T<\infty$, there exists a $H^1$-global solutions $(u,v,T)$ of \textnormal{(\ref{SNLSS})}. It holds $\Pas$ that
\begin{align*}
u,v\in \mathcal{Y}_{T},
\end{align*}
and uniqueness holds in the function space $\mathcal{Y}_T$, where $\mathcal{Y}_T$ is defined as \eqref{solspY}.
Also, the following uniform boundedness holds.
\[ \sup_{0\le t<T} (||u(t)||^2_{H^1}+||v(t)||^2_{H^1})<\infty, \ \Pas \]
\end{theorem}
\begin{proof}
Let $(u,v,(\tau_n)_{n\in\mathbb{N}},\tau^*(u_0,v_0))$ be the $H^1$-local solution of (\ref{SNLSS}). Then, by Lemma \ref{lem2} and Lemma \ref{lem56}, we have
\[ \mathbb{E} \sup_{0\le t<\tau^*(u_0,v_0)} (||u(t)||^2_{H^1}+||v(t)||^2_{H^1})\le \widetilde{C}(T)<\infty, \]
which implies
\[ \sup_{0\le t<\tau^*(u_0,v_0)} (||u(t)||^2_{H^1}+||v(t)||^2_{H^1})<\infty, \ \Pas \]
Then it follows from the blowup alternative in Theorem \ref{mainH} that $\tau^*(u_0,v_0)=T, \ \Pas$.
\end{proof}

We can remove the assumption about the noise coefficients $\mu_j$ for the existence of $H^1$-global solutions by applying the argument of the persistence of regularity. However, in this case, uniform boundedness cannot be obtained.
\begin{theorem}
\label{th541}
Assume (H1$)_1$, $1\le d\le3,$ and $\lambda,\kappa$ satisfy (\ref{49}). Then, for each $u_0,v_0\in H^1$ and $0<T<\infty$, there exists a $H^1$-global solution $(u,v,T)$ of \textnormal{(\ref{SNLSS})}. It holds $\Pas$ that
\begin{align*}
u,v\in \mathcal{Y}_{T},
\end{align*}
and uniqueness holds in the function space $\mathcal{Y}_T$.
\end{theorem}
By the equivalence of two expressions of solutions via the rescaling transformations (\ref{21}) and (\ref{22}),
Theorem \ref{th541} is rewritten as Theorem \ref{th551}.
\begin{theorem}
\label{th551}
Assume (H1$)_1$, $1\le d\le3,$ and $\lambda,\kappa$ satisfy (\ref{49}). Then, for each $u_0,v_0\in H^1$ and $0<T<\infty$, there exists a $H^1$-global solution $(y,z,T)$ of \textnormal{(\ref{RSNLSS})}. It holds $\Pas$ that
\begin{align*}
y,z\in \mathcal{Y}_{T},
\end{align*}
and uniqueness holds in the function space $\mathcal{Y}_T$.
\end{theorem}
\begin{proof}
From Theorem \ref{ymainH}, for any $u_0,v_0\in H^1$ and $0<T<\infty$, there exists a $H^1$-local solution $(y,z,\tau^*(u_0,v_0))$ where $\tau^*(u_0,v_0)$ is the maximum existence time of $(y,z)$. Assume that $\tau^*(u_0,v_0)<T$. Then, there exists $\varepsilon$ such that $\tau^*(u_0,v_0)-\varepsilon<\tau^*(u_0,v_0)+\varepsilon<T$. Set $I=[\tau^*(u_0,v_0)-\varepsilon,\tau^*(u_0,v_0)+\varepsilon]$. Fix $\omega\in\Omega$. We estimate the following map.
\begin{align*}
\mathcal{T}(y,z)(t)=&\left( \Phi(y,z),\Psi(y,z)\right) \\
:=&\left( U_{\ell}(t,0)u_0-i\int_0^tU_{\ell}(t,s)(\lambda\overline{e^{W(s)}}z(s)\overline{y(s)})ds, \right. \\
&\quad\left. U_L(t,0)v_0-i\int_0^tU_L(t,s)(\kappa e^{W(s)}y^2(s))ds \right).
\end{align*}
First, we consider the case $d=1$. Then, the space $\mathcal{Y}_{I}$ is $(C\cap L^{\infty})(I;H^1)\cap L^4(I;W^{1,\infty})$. We estimate $\Phi(y,z)$ and $\Psi(y,z)$ as
\begin{align*}
||\Phi(y,z)||_{\mathcal{Y}_{I}}\le & \  ||y(\tau^*(u_0,v_0)-\varepsilon)||_{H^1}+||z\overline{y}||_{L^1(I;H^1)}, \\
||\Psi(y,z)||_{\mathcal{Y}_{I}}\le & \  ||z(\tau^*(u_0,v_0)-\varepsilon)||_{H^1}+||y^2||_{L^1(I;H^1)},
\end{align*}
where
\begin{align*}
||z\overline{y}||_{L^1(I;H^1)}\lesssim& \ ||z\overline{y}||_{L^1(I;L^2)}+|||\nabla z||\overline{y}|||_{L^1(I;L^2)}+|||z||\nabla\overline{y}|||_{L^1(I;L^2)} \\
\lesssim& \ (2\varepsilon)^{\frac{3}{4}}||z||_{L^4(I;L^{\infty})}||y||_{L^{\infty}(I;L^2)}+(2\varepsilon)^{\frac{3}{4}}||\nabla z||_{L^4(I;L^{\infty})}||y||_{L^{\infty}(I;L^2)} \\
&+(2\varepsilon)^{\frac{3}{4}}||\nabla y||_{L^4(I;L^{\infty})}||z||_{L^{\infty}(I;L^2)} \\
\lesssim& \ (2\varepsilon)^{\frac{3}{4}}||z||_{L^4(I;W^{1,\infty})}||y||_{L^{\infty}(I;L^2)}+(2\varepsilon)^{\frac{3}{4}}||z||_{L^4(I;W^{1,\infty})}||y||_{L^{\infty}(I;L^2)} \\
&+(2\varepsilon)^{\frac{3}{4}}||y||_{L^4(I;W^{1,\infty})}||z||_{L^{\infty}(I;L^2)}.
\end{align*}
From Theorem 4.1, since $||y||_{L^{\infty}(I;L^2)},||z||_{L^{\infty}(I;L^2)}<\infty$, so $\Phi$ is bounded. In the same way, $\Psi$ is bounded. Differences are estimated as follows.
\begin{align*}
&||\Phi(y,z)-\Phi(y',z')||_{\mathcal{Y}_{I}}\lesssim  (2\varepsilon)^{\frac{3}{4}}C(T)(||y-y'||_{L^4(I;W^{1,\infty})}+||z-z'||_{L^4(I;W^{1,\infty})}), \\
&||\Psi(y,z)-\Psi(y',z')||_{\mathcal{Y}_{I}}\lesssim  (2\varepsilon)^{\frac{3}{4}}C(T)(||y-y'||_{L^4(I;W^{1,\infty})}+||z-z'||_{L^4(I;W^{1,\infty})}).
\end{align*}
Thus, if $\varepsilon$ is sufficiently small, $\Phi$ and $\Psi$ are contraction maps. Therefore, it contradicts the definition of the maximal existence time $\tau^*(u_0,v_0)$. The case of $d=2,3$ can be shown in the same way.
\end{proof}

\section*{Acknowledgement}
This work was supported by JSPS KAKENHI No. JP22J00787 (for M.H.), No. JP19H00644 and No. JP20K03671 (for S.M.).

\end{document}